\documentclass[10pt]{article}

\usepackage[english]{babel}
\usepackage{graphicx} 
\usepackage{amsthm} 
\usepackage{amsmath}
\usepackage{amssymb,verbatim}
\usepackage{amsthm} 
\usepackage{mathrsfs} 
\usepackage{bbm} 
\usepackage{calligra} 
\usepackage{enumerate} 
\usepackage[shortlabels]{enumitem}
\usepackage{xcolor}
\usepackage{multirow} 
\usepackage{subfigure}

\usepackage{hyperref}

\theoremstyle{theorem}
\newtheorem{thm}{Theorem}[section]
\newtheorem{proposition}{Proposition}[section]
\newtheorem{cor}{Corollary}[section]
\newtheorem{lemma}{Lemma}[section]
\theoremstyle{definition}
\newtheorem{defi}{Definition}[section]
\theoremstyle{remark}
\newtheorem{rem}{Remark}[section]

\newcommand{\N}{\mathbb{N}}

\newcommand{\R}{\mathbb{R}}
\newcommand{\C}{\mathbb{C}}
\newcommand{\T}{\mathbb{T}}

\newcommand{\pa}{\partial}

\newcommand{\car}[1]{\mathbbm{1}_{#1}}
\newcommand{\sgn}{\mathrm{sgn}}
\newcommand{\dif}{\,\mathrm{d}}

\newcommand{\Mixzone}{\Omega_{\mathrm{mix}}}

\def\Xint#1{\mathchoice
	{\XXint\displaystyle\textstyle{#1}}%
	{\XXint\textstyle\scriptstyle{#1}}%
	{\XXint\scriptstyle\scriptscriptstyle{#1}}%
	{\XXint\scriptscriptstyle\scriptscriptstyle{#1}}%
	\!\int}
\def\XXint#1#2#3{{\setbox0=\hbox{$#1{#2#3}{\int}$}
		\vcenter{\hbox{$#2#3$}}\kern-.5\wd0}}

\def\dashint{\Xint-}  

\numberwithin{equation}{section}

\allowdisplaybreaks[4]

\makeatletter
\renewcommand*\l@section{\@dottedtocline{1}{0em}{1.5em}}
\renewcommand*\l@subsection{\@dottedtocline{2}{1.5em}{2.3em}}
\renewcommand*\l@subsubsection{\@dottedtocline{3}{3.8em}{3.7em}}
\makeatother

\title{Localized mixing zone for Muskat bubbles and turned interfaces}
\author{\'A. Castro, D. Faraco, F. Mengual}
\date{}

\begin{document}

\maketitle

\begin{abstract}
We construct mixing solutions to the incompressible porous media equation starting from Muskat type data in the partially unstable regime. In particular, we consider bubble and turned type interfaces with Sobolev regularity. As a by-product, we prove the continuation of the evolution of IPM after the Rayleigh-Taylor and smoothness breakdown exhibited in \cite{CCFGL12,CCFG13}. At each time slice the space is split into three evolving domains: two non-mixing zones and a mixing zone which is localized in a neighborhood of the unstable region. In this way, we show the compatibility between the classical Muskat problem and the convex integration method.
\end{abstract}

\section{Introduction and main results}\label{sec:Intro}

We consider two incompressible
fluids 
with different constant densities $\rho_{-}$, $\rho_{+}$
and equal viscosity $\mu$,
separated by a connected curve $z^\circ=(z_1^\circ,z_2^\circ)$ inside a 2D porous medium with constant permeability $\kappa$ (or Hele-Shaw cell \cite{SaffmanTaylor58}) and under the action of gravity $-g (0,1)$.
As we deal with closed and open curves, it is convenient to fix an orientation for $z^\circ$.
For closed curves we fix the clockwise orientation ($\circlearrowright$) and for open curves the orientation from $x_1=-\infty$ to $+\infty$. Then, we denote $\Omega_{-}^\circ$ ($\Omega_{+}^\circ$) by the domain to the left (right) side of $z^\circ$. 
Thus, the initial density will be written as
\begin{equation}\label{rho0}
\rho^\circ(x):=
\left\lbrace
\begin{array}{rl}
\rho_{-}, & x\in\Omega_{-}^\circ,\\[0.1cm]
\rho_{+}, & x\in\Omega_{+}^\circ,
\end{array}\right.
\end{equation}
for $x=(x_1,x_2)\in\R^2$.
It is widely accepted that the dynamic of this two-phase flow can be modelled by the \textbf{IPM} (Incompressible Porous Media) system
\begin{align}
\partial_t\rho+\nabla\cdot(\rho v) & = 0, \label{IPM:1}\\
\nabla\cdot v & = 0,\label{IPM:2}\\
\tfrac{\mu}{\kappa}v & = -\nabla p-\rho g(0,1), \label{IPM:3}
\end{align}
where $\rho(t,x)$ $\equiv$ density, $v(t,x)$ $\equiv$ velocity field, $p(t,x)$ $\equiv$ pressure.
By normalizing, we may assume w.l.o.g.~that $|\rho_{\pm}|=\mu=\kappa=g=1$.

The investigations on the Muskat problem (\cite{Muskat34}) which deals with the interface evolution under the assumption of immiscibility, have been very intense both in the applied community due to the many applications (see e.g.~\cite{WoodingMorel76,TryggvasonAref83,Homsy87,ManickamHomsy95}) and in the theoretical side as this constitutes a challenging free boundary problem.

Mathematically, the theory has bifurcated into  two regimes, the so-called stable regime and unstable regime. 
This division arises from  the linear stability analysis of the equation for the interface evolution. It is classical (see e.g.~\cite{CCG11}) that such linear stability 
is characterized by  the sign of the Rayleigh-Taylor function $\sigma:=(\rho_{+}-\rho_{-})\partial_{\alpha}z_1^\circ$ as follows:
\begin{subequations}
\label{stability}
\begin{align}
\textrm{stable}\quad\textrm{on}&\quad \sigma(\alpha)>0,\label{stability:S}\\
\textrm{unstable}\quad\textrm{on}&\quad \sigma(\alpha)\leq 0.\label{stability:U}
\end{align}
\end{subequations}
This simply classifies whether the heavier fluid remains (locally) below the lighter one or not.
If the initial interface is a graph, $z^\circ(\alpha)=(\alpha,f^\circ(\alpha))$,
the interface evolution is governed by a 
a nonlinear parabolic equation, which can be linearized as $\partial_tf=(\rho_{+}-\rho_{-})(-\Delta)^{1/2}f$. Hence,
 the stability simply depends on the sign of the density jump $\rho_{+}-\rho_{-}$.  Therefore, 
for $\rho_{+}>\rho_{-}$ (i.e.~the heavier fluid is below $z^\circ$) what is called the fully stable
regime, the analogy with the heat equation gives hope of well-posedness theory in a suitable Sobolev space $H^k$. We refer to the corresponding weak solutions to  IPM as non-mixing solutions (see \cite{Yi03,SCH04,Ambrose04,CordobaGancedo07,CCG11} for initial results). In the last years there have taken extensive steps to reduce the initial  $k$ (see \cite{CGS16,CGSV17,Matioc19,AlazardLazar20,NguyenPausader20}). The current world record is the result of Alazard and Nguyen \cite{AlazardNguyenpp3} where they have proved the critical case $k=3/2$ (see also \cite{AlazardNguyenpp1,AlazardNguyenpp2}). For small enough initial data these solutions are global-in-time.
Additional results of global well-posedness for medium size initial data can be found in \cite{CCGRS16,CCGS13} and global solutions with large initial slope in \cite{CordobaLazar20,Cameron19,DLL17}.

The instability in the linearization is called  Rayleigh-Taylor (or Saffman-Taylor \cite{SaffmanTaylor58}) for the Muskat problem. In the graph case, it corresponds to  $\rho_{+}<\rho_{-}$ (i.e.~the heavier fluid is above $z^\circ$) what is called the fully unstable
regime, and the analogy is now with  the backwards heat equation. Therefore, it is to be expected that the problem is ill-posed unless the initial data is real-analytic $C^\omega$. As a matter of fact, all the techniques available in the stable case  catastrophic fail in this situation. Indeed, it can be proved that in the fully unstable regime, $\sigma(\alpha) \le 0$ for every $\alpha$,
the Cauchy problem for $f$ is ill-posed in Sobolev spaces (see e.g.~\cite{SCH04}). However,  practical and numerical experiments  show  the existence of the so-called   mixing solutions, solutions in which there exist a mixing zone where the two fluids mix stochastically (see e.g.~\cite{WoodingMorel76,Homsy87}). Numerically, it can be seen that small disturbances of an analytic initial interface increases rapidly creating finger patterns at different scales in the unstable region (see e.g.~\cite{TryggvasonAref83,ManickamHomsy95} and Figure \ref{fig:initial}).

In spite of the fact that the linearized problem is ill-posed and in accordance with what is observed in the experiments,
weak solutions to IPM, in the fully unstable case, have been constructed in the last years by replacing the continuum free boundary assumption  with the opening of a mixing zone $\Mixzone$ where the fluids begin to mix indistinguishably. These mixing solutions $(\rho,v)$ are recovered by the convex integration method applied in $\Mixzone$ to a so-called ``subsolution'' $(\bar{\rho},\bar{v},\bar{m})$ (cf.~Section \ref{sec:subsolution}). These subsolutions are intended to be a kind of coarse-grained solutions to IPM, with $\bar{m}$ representing the relaxation of the momentum $\bar{\rho}\bar{v}$.  The subsolutions are very
related to the relaxed solutions appearing in the Lagrangian relaxation approach of Otto \cite{Otto99,Otto01} (see also \cite{JKM21}).

In the context of large data, an striking result from  \cite{CCFGL12,CCFG13} shows that there
exist analytic initial interfaces in the fully stable regime (i.e. a graph) such that part of the curve
turns to the unstable regime (i.e. no longer a graph) and later, at some $T_* > 0$, the interface $z(T_*)$ is analytic but at a point in the unstable region where it is not $C^4$. The argument in \cite{CCFG13} could be adapted to prove weaker singularities in $C^k$ where $k\geq 5$ (i.e. the interface leaves to be $C^k$ but is still $C^{k-1}$). Thus, the Rayleigh-Taylor instability can arise spontaneously and the regularity might break down. After the blow-up time $T_*$ it is to be expected that the Muskat problem 
is ill-posed.

In this paper we give a method to construct mixing solutions to IPM in the Muskat partially unstable case. The original motivation was to continue the solutions after the breakdown
described in the previous paragraph. 
However,  there are numerous scenarios which are partially unstable. In this work we will  concentrate on two of them: The so-called {\bf bubble interfaces} where the two fluids are separated by a closed chord-arc curve
(see \cite{GGPSpp} for the case with surface tension) and the {\bf turned interfaces} where the interface is an open chord-arc curve which cannot be parametrized as a graph.
We describe both scenarios readily, prior to the statement of the theorems.
 
The bubble type initial interfaces are described by
\begin{equation}\label{Omega0}
\begin{split}
\Omega_-^\circ&\equiv\textrm{ exterior domain of }z^\circ,\\
\Omega_+^\circ&\equiv\textrm{ interior domain of }z^\circ,
\end{split}
\end{equation}
with $\rho_{\pm}=\pm 1$, for some closed chord-arc curve $z^\circ\in H^{k}(\T;\R^2)$ with $k$ big enough (cf.~Figure \ref{fig:initial}(a)).
Recall that we have taken $z^\circ$ clockwise oriented ($\circlearrowright$) to be consistent with the notation in \eqref{rho0}.

The turned type initial interfaces are described by
\begin{equation}\label{Omega0R}
\begin{split}
\Omega_-^\circ&\equiv\textrm{ upper domain of }z^\circ,\\
\Omega_+^\circ&\equiv\textrm{ lower domain of }z^\circ,
\end{split}
\end{equation}
with $\rho_{\pm}=\pm 1$, 
for some open chord-arc curve $z^\circ$ whose turned region $\{\partial_{\alpha}z_1^\circ(\alpha) \leq 0\}$ has positive measure. Here we consider both the $x_1$-periodic case
$z^\circ-(\alpha,0)\in H^k(\T;\R^2)$ and the asymptotically flat case $z^\circ-(\alpha,0)\in H^{k}(\R;\R^2)$ with $k$ big enough (cf.~Figure \ref{fig:initial}(b)).

Now we are ready to state our two main theorems.

\begin{thm}\label{thm:main:1}
For every closed chord-arc curve $z^\circ\in H^6(\T;\R^2)$ there exist infinitely many mixing solutions to IPM starting from \eqref{rho0}\eqref{Omega0} with $\rho_{\pm}=\pm 1$.
\end{thm}

\begin{thm}\label{thm:main:2}  For every open chord-arc curve $z^\circ$, either $x_1$-periodic $z^\circ-(\alpha,0)\in H^6(\T;\R^2)$ or asymptotically flat $z^\circ-(\alpha,0)\in H^{6}(\R;\R^2)$, whose turned region $\{\partial_{\alpha}z_1^\circ(\alpha) \leq 0\}$ has positive measure there exist infinitely many mixing solutions to IPM starting from \eqref{rho0}\eqref{Omega0R} with $\rho_{\pm}=\pm 1$.
\end{thm}
The definition of mixing solutions is by now classical and will be rigorously defined in Section \ref{sec:subsolution} where the reader is exposed to the convex integration framework.

\begin{rem} Theorem \ref{thm:main:2} is the first result proving the continuation of the evolution of IPM after the breakdown exhibited in \cite{CCFGL12,CCFG13}. 
\end{rem}

\begin{rem} As mentioned above, the h-principle applied to a coarse-grained solution, 
a subsolution, yields infinitely many weak solutions. This path goes in both directions as, by 
taking suitable averages of the solutions, the subsolution is essentially recovered \cite{CFM19}. Thus, the relevant macroscopic properties of the solutions are described by the subsolution. Our construction yields piecewise constant subsolutions as in \cite{ForsterSzekelyhidi18} and it is still open whether a continuous subsolution (similar to that in \cite{CCFpp}) might be built in the partially unstable regime.
 \end{rem}

\begin{rem}\label{rem:c}
As in \cite{Szekelyhidi12,CCFpp,ForsterSzekelyhidi18,NoisetteSzekelyhidi20}, our mixing zone grows linearly in time around an evolving pseudo-interface. However, in Theorems \ref{thm:main:1} and \ref{thm:main:2} the mixing region must be localized in a neighborhood of the unstable region. Furthermore, this approach reveals the admissible regime for the growth-rate $c(\alpha)$ of the mixing zone compatible with the relaxation of IPM. This is
\begin{equation}\label{c:regime:optimal}
\left| c(\alpha)+\frac{\sigma(\alpha)}{\sqrt{\sigma(\alpha)^2+\varpi(\alpha)^2}}\right|<1
\quad\textrm{on}\quad
c(\alpha)>0,
\end{equation}
which is characterized by the \textbf{Rayleigh-Taylor} function $\sigma:=(\rho_{+}-\rho_{-})\partial_{\alpha}z_1^\circ$ and the \textbf{vorticity} strength $\varpi:=-(\rho_{+}-\rho_{-})\partial_{\alpha}z_2^\circ$ along $z^\circ$ (cf.~Section \ref{sec:subsolution}). Observe that \eqref{c:regime:optimal} prevents the two fluids from mixing wherever the initial interface is stable ($\sigma(\alpha)>0$) and there is not vorticity ($\varpi(\alpha)=0$).
\end{rem}

\begin{figure}[h!]
	\centering
	\subfigure[A bubble type initial interface.]{\includegraphics[width=0.45\textwidth]{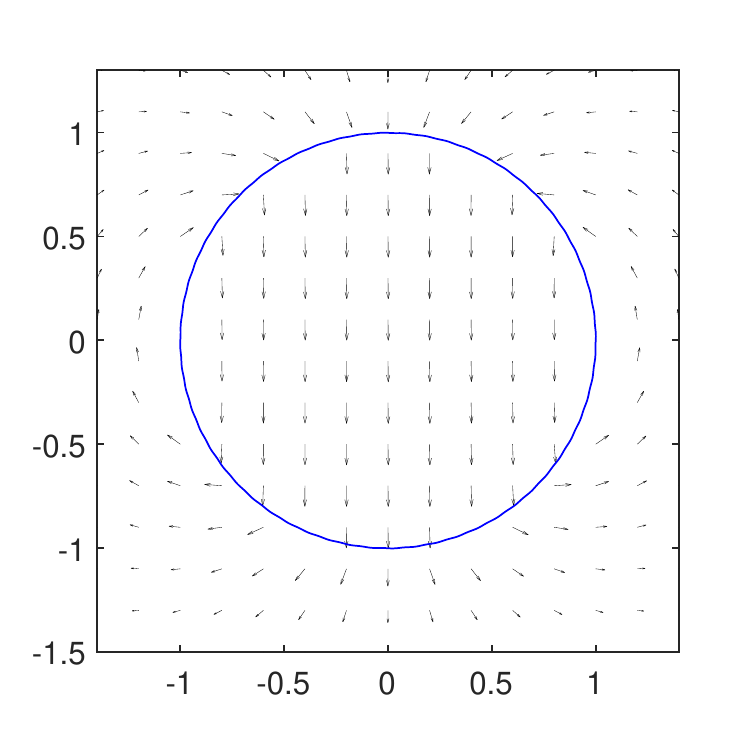}}\label{fig:initial:bubble}
	\subfigure[The localized mixing zone.]{\includegraphics[width=0.45\textwidth]{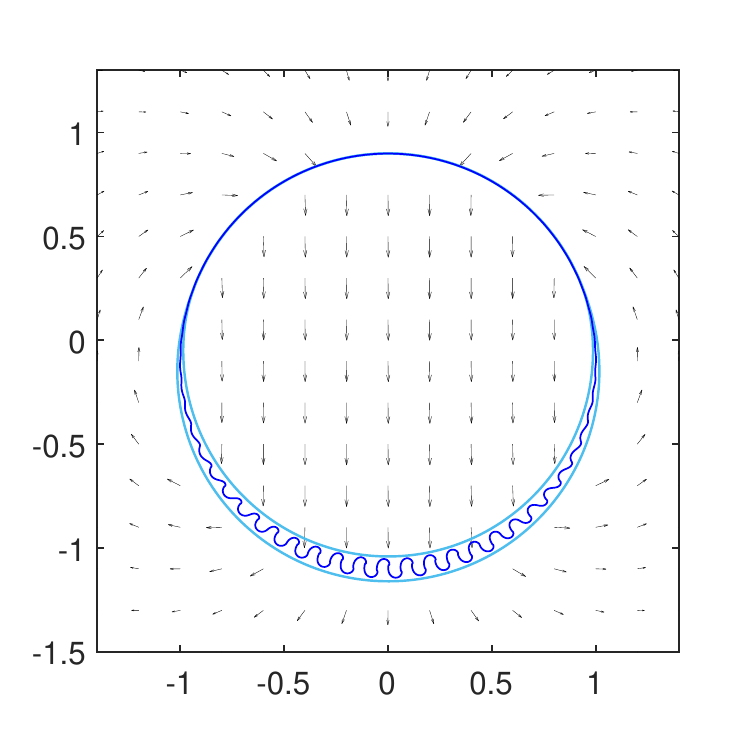}}
	\subfigure[A turned type initial interface.]{\includegraphics[width=0.45\textwidth]{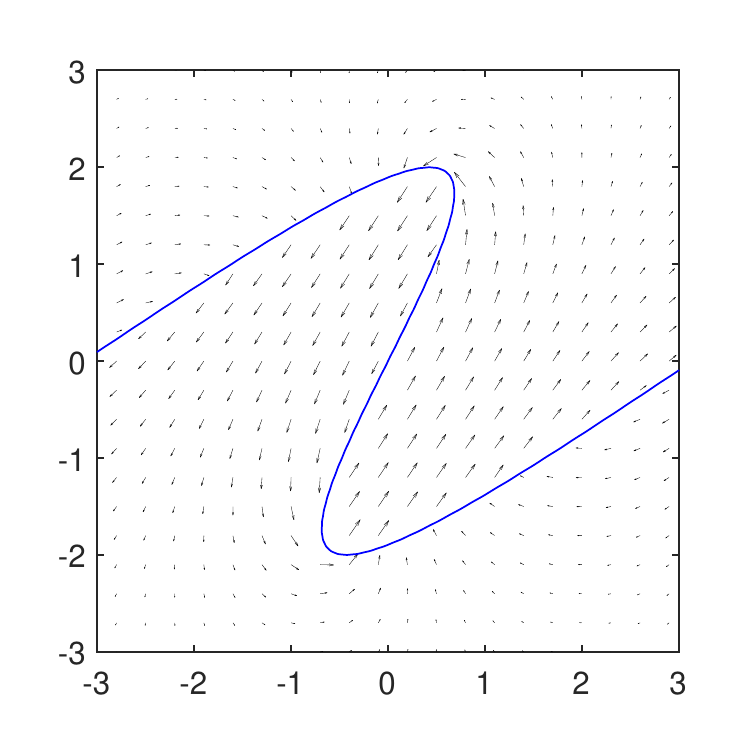}}\label{fig:initial:turned}
	\subfigure[The localized mixing zone.]{\includegraphics[width=0.45\textwidth]{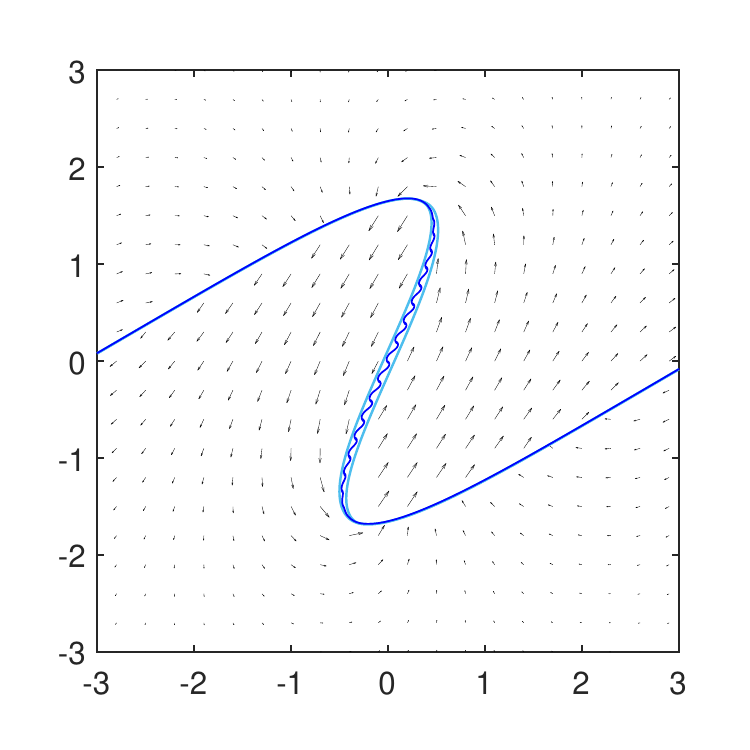}}
	\caption{(a)(c) The initial interface $z^\circ(\alpha)$ separating two fluids with different constant densities $\rho_{\pm}=\pm 1$ as in \eqref{Omega0}\eqref{Omega0R} respectively. (b)(d) At some $t>0$, the two boundaries of the non-mixing zones $z_{\pm}(t,\alpha)=z(t,\alpha)\mp tc(\alpha)\tau(\alpha)^\perp$ (light blue) for some pseudo-interface $z(t,\alpha)$ and growth-rate $c(\alpha)$, with $\tau(\alpha)=\frac{\partial_\alpha z^\circ(\alpha)}{|\partial_\alpha z^\circ(\alpha)|}$. Inside the mixing zone $\Mixzone(t)$ we plot the Rayleigh-Taylor curve $z_{\mathrm{per}}(t)$ (dark blue) which starts from a tiny perturbation of $z^\circ$ (via the vortex-blob method). In all the figures we have added the coarse-grained velocity field $\bar{v}(t,x)$ outside $\Mixzone$.}
	\label{fig:initial}
\end{figure}

The proof of the theorems rely on the pioneering adaptation of the convex integration method to Hydrodynamics by De Lellis and Sz\'ekelyhidi (\cite{DeLellisSzekelyhidi09,DeLellisSzekelyhidi10}). The method has turned out to be  very robust and flexible  and the research on it has been extremely intense in the last
 decade. We contempt ourselves with describing a few landmarks: It has successfully described several problems related to turbulence as the Onsager's conjecture (see e.g.~\cite{Isett18,BDSV19,DRSpp}),  the evolution of 
active scalars (\cite{CFG11,Shvydkoy11,Knott13,IsettVicol15,BSV19,HitruhinLindberg21}) and transport  equations (\cite{CGSW15,ModenaSzekelyhidi18,ModenaSzekelyhidi19,ModenaSattig20}), the
compressible Euler equations (see e.g.~\cite{ChiodaroliKreml14,CDK15,MarkfelderKlingenberg18,FKM20,AkramovWiedemann21,Markfelder21}), the Navier-Stokes equations (see e.g.~\cite{BuckmasterVicol19,BCV20,CDD18,BMSpp}) and Magnetohydrodynamics
(\cite{FaracoLindberg20,BBV20,FLS20}) (see also the surveys \cite{DeLellisSzekelyhidi12,DeLellisSzekelyhidi17,BuckmasterVicol19:0} and the references therein).

In the context of modeling instabilities in Fluid Dynamics via convex integration, the first result in the IPM context (see also \cite{CFG11}) was proved in \cite{Szekelyhidi12} where Sz\'ekelyhidi constructed infinitely many weak solutions to IPM starting from the unstable planar interface. Remarkably, the coarse-grained density (the subsolution in the convex integration jargon) agrees with the  Otto's Lagrangian relaxation of IPM (cf.~\cite{Otto99} and also \cite{Mengualpp}). In \cite{CCFpp} the first two authors and C\'ordoba constructed mixing solutions starting with a non-flat interface. In this work and all the subsequent ones, the mixing zone is described as an envelop of size $tc(\alpha)$ of a curve $z(t,\alpha)$ whose evolution is dictated by an operator which is an average of the classical Muskat operator. In \cite{CCFpp} the coarse-grained
density $\bar{\rho}$ is a continuous interpolation between the two fluids, which induces through an
adapted h-principle a degraded mixing property (\cite{CFM19}). As a by-product of this version of the  h-principle \cite{CFM19}, one shows that the subsolution is recovered from the solution by taking suitable averages. 
Remarkably, if one considers instead piecewise constant coarse-grained densities, the evolution of the pseudo-interface greatly simplifies  as was shown in \cite{ForsterSzekelyhidi18}  by F\"orster and Sz\'ekelyhidi.  See also \cite{ACF21,NoisetteSzekelyhidi20} for possible choices of the speed of opening of the mixing zone $c(\alpha)$. 

After the works in IPM, instabilities for the incompressible Euler equations have been successfully modeled with related  strategies, e.g.~the Rayleigh-Taylor (\cite{GKSpp,GebhardKolumbanpp}) and the Kelvin-Helmholtz (\cite{Szekelyhidi11,MengualSzekelyhidipp}) instabilities.\\

All the previous  works deal either with the fully stable
or fully  unstable regime of the various instabilities and hence new twists should be added
to the theory to deal with the partially unstable case. We finish the introduction with some comments on the natural obstructions and a non-technical description of the new view points
needed to address them. We believe that it is likely that the ideas from  this paper can be adapted and extended to consider different partially unstable scenarios in various problems
concerning instabilities in Fluid Dynamics.

Since it is to be expected that the classical Muskat problem is ill-posed in this partially unstable situation, we need to see a way to find compatibility between the parabolic analysis for the stable case and the relaxation approach for the unstable case. In particular, the mixing region needs to envelope  the unstable region. That is (recall $\sigma=(\rho_{+}-\rho_{-})\partial_{\alpha}z_1^\circ$) 
\begin{equation}\label{c:neig}
\{\sigma(\alpha)\leq 0\}\subset\{c(\alpha)>0\}.
\end{equation} 

As h-principles are by now standard \cite{Szekelyhidi12,CCFpp,CFM19}, the main issue
of the proof relies on building a mixing zone which admits a suitable subsolution $(\bar{\rho},\bar{v},\bar{m})$.
We will follow \cite{ForsterSzekelyhidi18} and declare $\bar{\rho}$ piecewise constant in the mixing
zone. In fact, for the sake of simplicity during the introduction we will assume the simplest case, $\bar{\rho}=0$ in $\Mixzone$.

At each time slice $0< t\leq T\ll 1$, the mixing zone is the open set in $\R^2$ given by
\begin{equation}\label{Mixzone}
\Mixzone(t):=\{z_{\lambda}(t,\alpha)\,:\,c(\alpha)>0,
\,\lambda\in (-1,1)\},
\end{equation}
parametrized by the map
\begin{equation}\label{Mixzone:z}
z_{\lambda}(t,\alpha)
:=z(t,\alpha)-\lambda t c(\alpha)\tau(\alpha)^\perp,
\end{equation}
where $\tau(\alpha)$ is an unitary vector field,
$c(\alpha)$ is the growth-rate of the mixing zone and $z(t,\alpha)$ is the pseudo-interface evolving from $z^\circ(\alpha)$, that we have to determine. 

In order to optimize the speed of opening of the mixing zone, it is convenient to take $\tau$ as the tangential vector field to $z^\circ$
\begin{equation}\label{t}
\tau(\alpha)
=\sgn(\rho_{+}-\rho_{-})
\frac{\partial_{\alpha}z^\circ(\alpha)}{|\partial_{\alpha}z^\circ(\alpha)|}.
\end{equation}

With our ansatz for $\bar{\rho}$ as in \cite{ForsterSzekelyhidi18} and this optimal choice for $\tau$, the admissible regime for $c(\alpha)$ compatible with the relaxation of IPM becomes
\begin{equation}\label{c:regime2}
\left| 2c(\alpha)+\frac{\sigma(\alpha)}{\sqrt{\sigma(\alpha)^2+\varpi(\alpha)^2}}\right|<1
\quad\textrm{on}\quad
c(\alpha)>0.
\end{equation}
We remark in passing that $2c(\alpha)$ above can be replaced by $\frac{2N}{2N-1}c(\alpha)$ for any $N\geq 1$ as in \cite{ForsterSzekelyhidi18,NoisetteSzekelyhidi20}, which yields \eqref{c:regime:optimal} as $N\to \infty$ (cf.~Section \ref{sec:Piecewise}).
Observe that this inequality requires $c(\alpha)=0$ if 
$\sigma(\alpha)=|(\sigma(\alpha),\varpi(\alpha))|$,
or equivalently $\partial_{\alpha}z_1^\circ(\alpha)=\sgn(\rho_{+}-\rho_{-})|\partial_{\alpha}z^\circ(\alpha)|$ (cf.~Remark \ref{rem:c}). 
Since in the regimes we are considering there are always such points, we are forced to treat the case where
there is no opening in some region, i.e.~$c(\alpha)=0$. 
An extra difficulty at this level
is that  our estimates need certain smoothness in $c$ (i.e.~the very definition of the velocity) which necessarily creates cusp singularities on $\Mixzone$. We deal with this problem by interpreting the mixing zone as a  superposition of regular domains (cf.~Figure \ref{fig:cusp} and Lemma \ref{lemma:parts}).

Next we turn to the 
coarse-grained velocity and the associated Muskat type operator. Here we start from \cite{ForsterSzekelyhidi18} as
we have chosen the same ansatz for the coarse-grained density and then explain the new idea. The F\"orster-Sz\'ekelyhidi's velocity is also an average of the classical Muskat velocity as in \cite{CCFpp} but only 
between the two boundaries of the non-mixing zones $z_\pm=z \mp t c \tau^\perp$. The associated Muskat type operator is (cf.~Section \ref{sec:Muskat})
 \begin{equation}\label{B:unstable}
B:=\frac{1}{2}\sum_{a=\pm}B_{a},
\quad\quad B_a:=\sum_{b=\pm}B_{a,b},
\end{equation}
where
\begin{equation}\label{Bab:unstable}
B_{a,b}(t,\alpha):=\frac{\rho_{+}-\rho_{-}}{4\pi}\int\left(\frac{1}{z_a(t,\alpha)-z_b(t,\beta)}\right)_1(\partial_{\alpha}z_{a}(t,\alpha)-\partial_{\alpha}z_{b}(t,\beta))\dif\beta.
\end{equation}
We remark in passing that, for open curves as in Theorem \ref{thm:main:2}, all these integrals are taken with the Cauchy's principal value at infinity. However, we will focus on the closed case until Section \ref{sec:generalizations} for clarity of exposition.

The evolution of $z$ is driven by the operator $B$.
On the one hand, as it is explained in the discussion after \eqref{c:regime2}, in the partially unstable case there is always a non-mixing region where we must solve a classical Muskat equation exactly
\begin{equation}\label{eq:z:stable}
\partial_tz=B
\quad\textrm{on}\quad
c(\alpha)=0.
\end{equation}
On the other hand, the flexibility of the notion of subsolution gives some space to define
the pseudo-interface (\cite{ForsterSzekelyhidi18,NoisetteSzekelyhidi20}). Namely, in 
the mixing region it is enough to solve \eqref{eq:z:stable} approximately 
$$
\partial_tz=B+\text{error}
\quad\textrm{on}\quad
c(\alpha)>0,
$$
where the error must be small in some sense that shall be specified in Sections \ref{sec:Muskat} and \ref{sec:conditionscz}. Due to the Rayleigh-Taylor instability, it is to be expected that the choice $\text{error}=0$ above yields an ill-posed equation as in the fully unstable regime. In spite of this, following another clever idea from \cite{ForsterSzekelyhidi18}, 
in the fully unstable regime it is possible to take
$\text{error}=B^{1)}-B+\text{error}$, where $B^{1)}$ denotes the first order expansion in time of $B$. This choice yields the following well-defined evolution for $z$
\begin{equation}\label{eq:z:unstable}
\partial_tz=B^{1)}+\text{error}
\quad\textrm{on}\quad
c(\alpha)>0.
\end{equation}
We remark that, if the error in \eqref{eq:z:unstable} was zero, then the equations \eqref{eq:z:stable} and \eqref{eq:z:unstable} do not match at $c(\alpha)=0$. 
In order to glue these equations we first introduce a partition of the unity $\{\psi_0,\psi_1\}$ which, as required in \eqref{c:neig}, allows also to open the mixing zone slightly inside the stable region, namely $\mathrm{supp}\,\psi_0\subset\{\partial_{\alpha}z_1^\circ(\alpha)>0\}$ and $\mathrm{supp}\,\psi_1=\mathrm{supp}\,c$. 
That is,  we bypass the gluing problem by writing
$$\partial_tz=\psi_0 B+ \psi_1 B^{1)}+\text{error}
\quad\textrm{on}\quad\T,$$
where the error is supported on $\{c(\alpha)>0\}.$
Yet the energy inequalities that we obtain for the operator $\psi_0B$ (or other modifications) yields a factor $1/c$ which blows up in the region where $c(\alpha)$ tends to zero. The  way out of this vicious circle  is to treat the interaction between separate boundaries as a perturbation.
In this way, one can write 
$B=E+\text{error}$ in such a way that $E$ yields good energy inequalities and the $\text{error}$ is small in the supremum norm and supported on $\{c(\alpha)>0\}$.
Thus, the perturbation can be absorbed in the relaxation even if its derivatives are badly behaving. Hence, we will solve
$$\partial_t z=\psi_0 E+\psi_1 E^{1)}+\text{error}
\quad\textrm{on}\quad\T,$$ 
for some error term supported on $\{c(\alpha)>0\}$, where $E^{1)}$ denotes the first order expansion in time of $E$. Essentially, $E=B_{+,+}+B_{-,-}$ as the factor $1/c$ comes from the terms with $a\neq b$ in \eqref{Bab:unstable}.\\

\textbf{Organization of the paper}. 
We start Section 2 by recalling briefly the
Classical and the Mixing Muskat problem. After this, 
we recall also the concepts of mixing solution and subsolution, as well as the h-principle in IPM. Then, we define our ansatz for the subsolution in terms of the mixing zone and derive the conditions for the growth-rate $c$ and the pseudo-interface $z$ under which such subsolution truly exists. The construction of a pair $(c,z)$ satisfying such requirements appears in Sections \ref{sec:c}-\ref{sec:existence}. Finally, we prove in Section \ref{sec:generalizations} the Theorems \ref{thm:main:1}, \ref{thm:main:2} and the optimal regime for $c$ given in \eqref{c:regime:optimal}.\\

\textbf{Notation}.

\begin{itemize}
	\item (Complex coordinates)
	It is convenient to identify the Euclidean space $\R^2$ with the complex plane $\C$ as usual, $z=(z_1,z_2)=z_1+iz_2$. Therefore, along the whole paper we will use complex coordinates
	and subindexes $1,2$ indicate real and imaginary parts for a complex number.
	
	Thus, $i\equiv (0,1)$ plays the roll both of the standard vertical vector and the imaginary unit.
	We will denote $z^*:=(z_1,-z_2)=z_1-iz_2$, $z^\perp:=(-z_2,z_1)=iz$ and $z\cdot w:=z_1w_1+z_2w_2=(zw^*)_1$. In this regard, we also have $\nabla=(\partial_1,\partial_2)=\partial_1+i\partial_2$, and so $\nabla^*=\partial_1-i\partial_2$ and $\nabla^\perp=i\nabla$.
	
	\item (Function spaces) We will consider the usual H\"older spaces $C^{k,\delta}$ with norm
	$$\|f\|_{C^{k,\delta}}:=\sup_{j\leq k}\|\partial^jf\|_{L^\infty}+|\partial^kf|_{C^\delta}
	\quad\textrm{with}\quad
	|g|_{C^\delta}:=\sup_{\alpha,\beta}\frac{|g(\alpha)-g(\alpha-\beta)|}{|\beta|^\delta},$$
	and also the Sobolev spaces $H^{k}$ with 
	$$\|f\|_{H^k}:=\left(\sum_{j=0}^k\|\partial^jf\|_{L^2}^2\right)^{\frac{1}{2}}.$$
	
	
	\item (Increments and different quotients) Given a function $f=f(\alpha)$ and another
	parameter $\beta$, it will be handy to use the expressions
	$f'=f(\alpha-\beta)$, $\delta_\beta f=f-f'$ and $\Delta_\beta=\frac{\delta_\beta}{\beta}$
	as in \cite{CCG18,CordobaLazar20}.
	\end{itemize}

\section{The mixing zone and the subsolution}\label{sec:subsolution}

\subsection{The Muskat Problem}\label{sec:Muskat}
The Muskat problem describes IPM under the assumption that there is a time-dependent
oriented curve $z(t,\alpha)$ separating $\R^2$ into two complementary open domains
\begin{equation}\label{Omega:Muskat}
\begin{split}
\Omega_{-}(t)\,&\equiv\textrm{ domain to the left side of }z(t),\\
\Omega_{+}(t)\,&\equiv\textrm{ domain to the right side of }z(t),
\end{split}
\end{equation}
each one occupied by a fluid with different constant densities $\rho_{-}$ and $\rho_{+}$ respectively.

The incompressibility condition \eqref{IPM:2} implies that $v=\nabla^\perp\psi$ for some stream function $\psi(t,x)$.
Hence, the Darcy's law \eqref{IPM:3} can be written in complex coordinates as
$\nabla(p+i\psi)=-i\rho,$
which yields the following Poisson equation ($\nabla^*\nabla=\Delta$)
$$\Delta(p+i\psi)=-i\nabla^*\rho.$$
In view of \eqref{Omega:Muskat}, the density jump along $z$ implies that
$$\nabla\rho=-(\rho_{+}-\rho_{-})\partial_{\alpha}z^\perp\delta_z,$$
in the sense of distributions.
Hence, $p$ and $\psi$ are recovered from the Poisson equation through the Newtonian potential
$$(p+i\psi)(t,x)
=\frac{\rho_{+}-\rho_{-}}{2\pi}\int\log|x-z(t,\beta)|\partial_{\alpha}z(t,\beta)^*\dif\beta,
\quad\quad
x\neq z(t,\beta).$$
Then, $p$ and $\psi$ are continuous but have discontinuous gradients along $z$, and indeed $\Delta(p+i\psi)=(\sigma+i\varpi)\delta_z$ where $\sigma$ $\equiv$ Rayleigh-Taylor and $\varpi$ $\equiv$ vorticity strength ($\Delta\psi=\nabla^\perp\cdot v=\omega$), which satisfy
\begin{equation}\label{sigmavarpi}
\sigma+i\varpi=(\rho_{+}-\rho_{-})\partial_{\alpha}z^*.
\end{equation}
The velocity
$v$ is recovered from the vorticity through the Biot-Savart law
\begin{equation}\label{v:stable}
\begin{split}
v(t,x)
&=\left(\frac{1}{2\pi i}\int\frac{\varpi(t,\beta)}{x-z(t,\beta)}\dif\beta\right)^*\\
&=-\frac{\rho_{+}-\rho_{-}}{2\pi }\int\left(\frac{1}{x-z(t,\beta)}\right)_1\partial_{\alpha}z(t,\beta)\dif\beta,
\quad\quad
x\neq z(t,\beta),
\end{split}
\end{equation}
where we have applied that $\varpi=-(\rho_{+}-\rho_{-})\partial_{\alpha}z_2$ and the Cauchy's argument principle  in the last equality
\begin{equation}\label{CAP:stable}
\left(\int\frac{\partial_{\alpha}z(t,\beta)}{x-z(t,\beta)}\dif\beta\right)_1=0,
\quad\quad
x\neq z(t,\beta).
\end{equation}
It is easy to see that $v$ is bounded, smooth outside $z$ but with tangential discontinuities along $z$. Its normal component is well-defined and satisfies
$$\lim_{\Omega_{\pm}(t)\ni x\to z(t,\alpha)}(v(t,x)-B(t,\alpha))\cdot\partial_{\alpha}z(t,\alpha)^\perp
=0,$$
where
\begin{equation}\label{B:stable}
B(t,\alpha):=\frac{\rho_{+}-\rho_{-}}{2\pi}\int\left(\frac{1}{z(t,\alpha)-z(t,\beta)}\right)_1(\partial_{\alpha}z(t,\alpha)-\partial_{\alpha}z(t,\beta))\dif\beta.
\end{equation}
Observe that the operator $B$ is obtained by adding a suitable tangential term to the velocity \eqref{v:stable} and then taking the limit $\Omega_{\pm}(t)\ni x\rightarrow z(t,\alpha)$. We refer to \eqref{B:stable} as the classical Muskat operator.
Let us remark that this operator \eqref{B:stable} coincides with \eqref{B:unstable} when $tc$ is identically zero. Since it only appears in this Subsection \ref{sec:Muskat} and the notation of the paper is heavy enough, we do not give it another name.

Finally, it is easy to check that the conservation of mass equation \eqref{IPM:1} is equivalent to find $z$ satisfying \begin{equation}\label{eq:z:stable:0}
(\partial_tz-B)\cdot\partial_{\alpha}z^\perp=0.
\end{equation}
Thus, the Muskat problem is equivalent to solve this Cauchy problem for the interface $z$ starting from $z^\circ$ given in \eqref{eq:z:stable}. 
We remark that because of \eqref{eq:z:stable:0}, one may add any tangential term to \eqref{eq:z:stable}. This only changes the parametrization and does not modify the geometric evolution of the curve. We refer to \eqref{eq:z:stable} as the Classical Muskat problem.

Assuming that the interface can be parametrized as a graph, $z(t,\alpha)=\alpha+if(t,\alpha)$ in complex coordinates,   
the equation \eqref{eq:z:stable} reads as
$$
\partial_tf=\frac{\rho_{+}-\rho_{-}}{2\pi}\mathrm{pv}\!\int_{\R}\left(\frac{1}{1+i\triangle_\beta f}\right)_1\partial_{\alpha}\triangle_\beta f\dif\beta,
$$
which can be linearized as $\partial_tf=(\rho_{+}-\rho_{-})(-\Delta)^{1/2}f$. In analogy with the heat equation, the fully stable regime ($\rho_+>\rho_-$) admits a parabolic analysis through energy estimates. 

However, the same strategy for the fully unstable regime ($\rho_+<\rho_-$) is not viable.
Despite this, mixing solutions to IPM starting from fully unstable Muskat inital data have been constructed in the last years through the convex integration method \cite{Szekelyhidi12,CCFpp,ForsterSzekelyhidi18,NoisetteSzekelyhidi20}. In these works, the mixing zone is given as in \eqref{Mixzone}\eqref{Mixzone:z}
but with $\tau=(-1,0)$ instead of \eqref{t}. More generally, we may consider
any unitary vector field $\tau(\alpha)$ satisfying 
$$(\rho_{+}-\rho_{-})\partial_\alpha z^\circ(\alpha)\cdot\tau(\alpha) > 0
\quad\textrm{on}\quad
c(\alpha)>0.$$
Thus, the triplet $(\tau,c,z)$ parametrizes the mixing zone, which does not exist when $tc(\alpha)=0$.
Here we follow \cite{ForsterSzekelyhidi18,NoisetteSzekelyhidi20}, where  $\bar{\rho}=0$ on $\Mixzone$. In this case, the coarse-grained velocity becomes
$$\bar{v}(t,x)=-\frac{\rho_{+}-\rho_{-}}{4\pi }\sum_{b=\pm}\int\left(\frac{1}{x-z_b(t,\beta)}\right)_1\partial_{\alpha}z_b(t,\beta)\dif\beta,
\quad\quad 
x\neq z_b(t,\beta),
$$
where $z_\pm(t,\alpha)=z(t,\alpha)\mp tc(\alpha)\tau(\alpha)^\perp$ are the two boundaries of the non-mixing zones.
The admissible regime for $c(\alpha)$ compatible with the relaxation of IPM is
\begin{equation}\label{c:regime}
\left| 2c(\alpha)+\frac{\sigma(\alpha)}{(\rho_{+}-\rho_{-})\partial_{\alpha}z^\circ(\alpha)\cdot\tau(\alpha)}\right|<1
\quad\textrm{on}\quad
c(\alpha)>0,
\end{equation}
which agrees with \cite{ForsterSzekelyhidi18,NoisetteSzekelyhidi20} as in this case $\rho_{\pm}=\mp 1$, $\partial_{\alpha}z^\circ=(1,\partial_{\alpha}f^\circ)$ and $\tau=(-1,0)$ (cf.~Rem.~\ref{rem:ctau}).
Observe that \eqref{c:regime} requires $c(\alpha)=0$  if $\partial_{\alpha}z_1^\circ(\alpha)=\partial_{\alpha}z^\circ(\alpha)\cdot\tau(\alpha)$. 
In view of \eqref{c:neig}, this prevents some choices for $\tau(\alpha)$ as for instance the one from \cite{CCFpp,ForsterSzekelyhidi18,CFM19}.
Thus, we really need to optimize by opening the mixing zone perpendicularly to the curve
(this is also the case in \cite{MengualSzekelyhidipp}).  This is why we have chosen $\tau$ as in \eqref{t}.
With such optimal  choice for $\tau$, \eqref{c:regime} reads as \eqref{c:regime2} (recall \eqref{sigmavarpi}).

Once $\tau(\alpha)$ and $c(\alpha)$ are fixed, we must determine the time-dependent pseudo-interface $z(t,\alpha)$.
Modulo technical details which will be explained in Section \ref{sec:conditionscz}, the existence of a relaxed momentum $\bar{m}(t,x)$ is reduced to find $z$ satisfying
\begin{equation}\label{eq:z:partially:0}
\int_{0}^{\alpha}\left((\partial_tz-B)\cdot\partial_{\alpha}z^\perp+tD\cdot\partial_{\alpha}(c\tau)\right)\dif\alpha'=o(t)c(\alpha),
\end{equation}
uniformly in $\alpha$ as $t\to 0$, where
\begin{equation}\label{def:D}
D(t,\alpha):=-\frac{1}{2}\sum_{a=\pm}aB_{a}-i(c\tau+\tfrac{1}{2}),
\end{equation}
with $\tau$, $c$, $B$ and $B_{a}$ given in \eqref{Mixzone}-\eqref{Bab:unstable}.
Observe that $D\cdot\partial_{\alpha}(c\tau)=0$ for $\partial_{\alpha}z=(1,\partial_{\alpha}f)$ and $\tau=(-1,0)$,
and thus it does not appear in \cite{ForsterSzekelyhidi18,NoisetteSzekelyhidi20}. Hence, the equation \eqref{eq:z:partially:0} generalizes both \eqref{eq:z:stable} and \eqref{eq:z:unstable}.
As we mentioned in the introduction, we cannot simply glue these evolution equations because they do not match at $c(\alpha)=0$.
In order to interpolate between the two regions, we introduce a partition of the unity $\{\psi_0,\psi_1\}$ subordinated to  $\{\partial_{\alpha}z_1^\circ(\alpha)>0\}$ and $\{c(\alpha)>0\}$ respectively. 
Then, we consider (cf.~\eqref{eq:z})
$$
\partial_tz=\psi_0E+\psi_1 E^{1)}+\text{error}.
$$
Here,  $E$ should be an extension of $B$ and good for energy inequalities, which justify
its name twice. 

In view of \eqref{eq:z:stable} and \eqref{eq:z:unstable}, 
one would be initially tempted to take $E=B$. However, the terms with $a\neq b$ in \eqref{B:unstable} introduce a factor $\partial_{\alpha}\log c(\alpha)$ in the energy estimates which we did not see how to compensate. Thus, we will declare
\begin{equation}\label{def:E2} 
E=\sum_{b=\pm}B_{b,b},
\end{equation}
which equals $B$ on $tc(\alpha)=0$ and only includes interaction of stable Muskat type.

The error term is localized on the mixing region with order $t$. This is
$$\text{error}=-(t\kappa+i(tD^{(0)}\cdot\partial_{\alpha}(c\tau)+h\psi_1)\partial_{\alpha}z^{\circ}),$$
where $D^{(0)}=D|_{t=0}$, $\kappa=\partial_t(E-B)|_{t=0}$ depends on the initial curvature and $h=O(t^2)$ is a time-dependent average. As a result,
$D^{(0)}$ and $\kappa$ only depends on $z^\circ$ while $h(t)$ depends on $z(t)$ but not on $\alpha$. 
This allows to treat the error as a harmless term in the energy estimates.

\subsection{Weak solutions, subsolutions and the mixing zone}\label{sec:Mixzone}

Let us start by recalling the rigorous definition of weak solutions, mixing solutions
and subsolutions in the IPM context. 

Given $T>0$ and $\rho^\circ$ as in \eqref{rho0}, a \textbf{weak solution} to IPM
$$(\rho,v)\in C([0,T];L_{w^*}^\infty(\R^2;[-1,1]\times\R^2))$$
satisfies that,  for every test function $\phi\in C_c^1(\R^3)$ with $\phi^\circ:=\phi|_{t=0}$ and $0<t\leq T$:
\begin{subequations}\label{IPM:weak}
	\begin{align}
	\int_0^t\int_{\R^2}\rho(\partial_t\phi+v\cdot\nabla\phi)\dif x\dif s
	&=\int_{\R^2}\rho(t)\phi(t)\dif x-\int_{\R^2}\rho^\circ\phi^\circ\dif x,\label{IPM:weak:1}\\
	\int_0^t\int_{\R^2}v\cdot\nabla\phi\dif x\dif s&=0,\label{IPM:weak:2}\\
	\int_0^t\int_{\R^2}(v+\rho i)\cdot\nabla^\perp\phi\dif x\dif s&=0.\label{IPM:weak:3}
	\end{align}
\end{subequations}
In addition, a weak solution is a \textbf{mixing solution} if, at each $0<t\leq T$, the space $\R^2$ is split into three complementary open domains, $\Omega_{+}(t)$, $\Omega_{-}(t)$ and $\Mixzone(t)$, satisfying that
$(\rho,v)$ is continuous on the non-mixing zones $\Omega_\pm$:
\begin{equation}\label{eq:non-mixing}
\rho=\pm 1\quad\textrm{on}\quad\Omega_{\pm},
\end{equation}
while it behaves wildly inside the mixing zone $\Mixzone$:
\begin{equation}\label{eq:mixing}
\int_{\Omega}(1-\rho^2)\dif x=0
<\int_{\Omega}(1-\rho)\dif x
\int_{\Omega}(1+\rho)\dif x,
\end{equation}
for every open $\emptyset\neq\Omega\subset\Mixzone(t)$.\\
Conversely, we say that $(\rho,v)$ is a \textbf{non-mixing solution} if $\Mixzone=\emptyset$.\\

In convex integration, a subsolution (a macroscopic solution) is defined in term of a conservation law and a relaxed constitutive relation, which is typically given by the $\Lambda$-convex hull. In the IPM context, the hull was computed in \cite{Szekelyhidi12} (see also \cite{Mengualpp} for related computations).

Given $T>0$ and $\rho^\circ$ as in \eqref{rho0}, a \textbf{subsolution} to IPM
$$(\bar{\rho},\bar{v},\bar{m})\in C([0,T];L_{w^*}^\infty(\R^2;[-1,1]\times\R^2\times\R^2))$$
satisfies that, for every test function $\phi\in C_c^1(\R^3)$ with $\phi^\circ:=\phi|_{t=0}$ and $0< t\leq T$:
\begin{subequations}
\label{RIPM}
\begin{align}
\int_0^t\int_{\R^2}(\bar{\rho}\partial_t\phi+\bar{m}\cdot\nabla\phi)\dif x\dif s
&=\int_{\R^2}\bar{\rho}(t)\phi(t)\dif x-\int_{\R^2}\rho^\circ\phi^\circ\dif x, \label{RIPM:1}\\
\int_0^t\int_{\R^2}\bar{v}\cdot\nabla\phi\dif x\dif s&=0,\label{RIPM:2}\\
\int_0^t\int_{\R^2}(\bar{v}+\bar{\rho} i)\cdot\nabla^\perp\phi\dif x\dif s&=0, \label{RIPM:3}
\end{align}
\end{subequations}
such that, at each $0<t\leq T$, the space $\R^2$ is split into three complementary open domains, $\Omega_{+}(t)$, $\Omega_{-}(t)$ and $\Mixzone(t)$ satisfying that
\begin{subequations}
\label{hull}
\begin{align}
\bar{\rho}=\pm 1,\quad\bar{m}=\bar{\rho}\bar{v}
&\quad\textrm{on}\quad\Omega_{\pm},\label{hull:1}\\
|2(\bar{m}-\bar{\rho}\bar{v})+(1-\bar{\rho}^2)i|<(1-\bar{\rho}^2)
&\quad\textrm{on}\quad\Mixzone.\label{hull:2}
\end{align}
\end{subequations}
In addition, it is required that
\begin{equation}\label{velocity:bdd}
\sup_{0\leq t\leq T}\|\bar{v}(t)\|_{L^\infty}<\infty.
\end{equation}

\begin{rem}
Notice that the pressure does not appear in \eqref{IPM:weak:3}\eqref{RIPM:3}. For completeness we will show in Lemma \ref{lemma:p} how $p\in C([0,T]\times \R^2)$ is recovered and its relation with $\bar{p}$. 
We are not aware of similar computations for the IPM pressure in the convex integration framework.
\end{rem}

\begin{thm}[H-principle in IPM]\label{thm:hprinciple}
Assume that there exists a subsolution $(\bar{\rho},\bar{v},\bar{m})$ to IPM starting from $\rho^\circ$, for some $T>0$, $\Omega_\pm$ and $\Mixzone$. Then, there exist infinitely many mixing solutions $(\rho,v)$ to IPM starting from $\rho^\circ$, for the same $T>0$, $\Omega_{\pm}$ and $\Mixzone$, and satisfying $(\rho,v)=(\bar{\rho},\bar{v})$ outside $\Mixzone$.
\end{thm}

The proof of this h-principle 
for the $L_{t,x}^\infty$ case can be found in  \cite{Szekelyhidi12}, and the generalization to $C_tL_{w^*}^\infty$ in \cite{CFM19}. As noticed in \cite{Szekelyhidi12}, the inequality \eqref{hull:2} only provides solutions in $L^2$. Remarkably, Sz\'ekelyhidi computed in \cite[Prop.~2.4]{Szekelyhidi12} the additional inequalities which yield solutions in $L^\infty$. In \cite[Lemma 4.3]{Mengualpp}  
it is checked that, if $\bar{v}$ is controlled as in \eqref{velocity:bdd}, then these additional inequalities are automatically satisfied.\\

By Theorem \ref{thm:hprinciple}, the construction of mixing solutions as stated in Theorems \ref{thm:main:1} and \ref{thm:main:2} is reduced to constructing
suitable subsolutions $(\bar{\rho},\bar{v},\bar{m})$ adapted to $\Mixzone$. 

As described in the intro, it follows from  \eqref{Mixzone}-\eqref{t} that the mixing zone is prescribed 
by the \textbf{growth-rate} $c$ and the \textbf{pseudo-interface} $z$. With this terminology,
 for bubble interfaces \eqref{Omega0} we define
\begin{equation}\label{Omegat}
\begin{split}
\Omega_-(t)&\equiv\textrm{ exterior domain of }z_-(t),\\
\Omega_+(t)&\equiv\textrm{ interior domain of }z_+(t),
\end{split}
\end{equation}
and for turned interfaces \eqref{Omega0R}
\begin{equation}\label{OmegatR}
\begin{split}
\Omega_-(t)&\equiv\textrm{ upper domain of }z_-(t),\\
\Omega_+(t)&\equiv\textrm{ lower domain of }z_+(t).
\end{split}
\end{equation}
Thus,  $\Omega_{+}(t)$, $\Omega_{-}(t)$ and $\Mixzone(t)$ are complementary open domains in $\R^2$. For each region $r=+,-,\mathrm{mix}$, we denote $\Omega_r:=\{(t,x)\,:\,x\in\Omega_{r}(t),\,0\leq t\leq T\}$.

\begin{rem}
For the sake of simplicity we will consider from now on the closed case \eqref{Omegat} and go back at Section \ref{sec:generalizations} with the open case \eqref{OmegatR}. Recall that for closed interfaces we have assumed that $z^\circ$ is clockwise oriented ($\circlearrowright$).
In addition, we may assume w.l.o.g.~that $z^\circ$ is the arc-length ($|\partial_{\alpha}z^\circ|=1$) parametrization, although $z(t)$ will not be it in general. Thus, we fix $\T=[-\ell_\circ/2,\ell_\circ/2]$ where $\ell_\circ:=\mathrm{length}(z^\circ)$. 
\end{rem}

For a general pair $(c,z)$ we construct in the next Section \ref{sec:subsol} a suitable triplet $(\bar{\rho},\bar{v},\bar{m})$ adapted to $\Mixzone$. After this, we derive in Section \ref{sec:conditionscz} conditions for $(c,z)$ under which this $(\bar{\rho},\bar{v},\bar{m})$ becomes a subsolution. Finally, we will prove in Sections \ref{sec:c}-\ref{sec:existence} the existence of a pair $(c,z)$ satisfying such requirements.\\ 

 
\textbf{Hipothesis on $(c,z)$}.
Apart from regularity assumptions, the curve needs to satisfy an angle and a chord-arc condition in a 
uniform manner. Prior to state the assumptions, 
let us introduce the angle constant of $z$ w.r.t.~$\tau$ 
\begin{equation}\label{A}
\mathcal{A}(z)
:=\inf\left\{\frac{\partial_{\alpha}z(\alpha)}{|\partial_{\alpha}z(\alpha)|}\cdot\tau(\alpha)
\,:\,\alpha\in\T\right\},
\end{equation}
and
recall the definition of a chord-arc curve. 

\begin{defi}\label{defi:CA}
A curve $z\in C(\T;\R^2)$ is \textbf{chord-arc} if
\begin{equation}\label{CA}
\mathcal{C}(z):=\sup\left\{\left|\frac{\beta}{z(\alpha)-z(\alpha-\beta)}\right|\,:\,\alpha,\beta\in\T\right\}<\infty.
\end{equation}
\end{defi}

Along the rest of this Section \ref{sec:subsolution} we will  assume the existence of $\delta,T>0$ such that
\begin{equation}\label{hyp:c}
c\in C^{1,\delta}(\T),
\quad\quad c\geq 0,
\end{equation}
and 
\begin{equation}\label{hyp:z}
z\in C^1([0,T];C^{1,\delta}(\T;\R^2)),
\quad\quad
z|_{t=0}=z^\circ\in C^{2,\delta}(\T;\R^2),
\end{equation}
satisfying the following equi-angle condition
\begin{equation}\label{Angle}
\mathcal{A}(c,z)
:=\inf\left\{\frac{\partial_{\alpha}z_\lambda(t,\alpha)}{|\partial_{\alpha}z_\lambda(t,\alpha)|}\cdot\tau(\alpha)
\,:\,\alpha\in\T,\,\lambda\in[-1,1],\,0\leq t\leq T\right\}>0,
\end{equation}
and the following equi-chord-arc condition
\begin{equation}\label{equiCA}
\mathcal{C}(c,z)
:=\sup\left\{\frac{\sqrt{\beta^2+((\lambda-\mu)tc(\alpha))^2}}{|z_{\lambda}(t,\alpha)-z_{\mu}(t,\alpha-\beta)|}
\,:\,\alpha,\beta\in\T,\,\lambda,\mu\in[-1,1],\,0\leq t\leq T\right\}<\infty,
\end{equation}
where we recall that $z_\lambda=z-\lambda tc\tau^\perp$ with $\tau=\partial_{\alpha}z^\circ$.\\

The condition \eqref{Angle} controls the angle between the family of curves $z_\lambda$ w.r.t.~$\tau$.
The equi-chord-arc condition \eqref{equiCA} bounds the singularity due to the denominator of the operators $B_{a,b}$ \eqref{Bab:unstable}, while the numerator justifies the regularity assumptions \eqref{hyp:c}\eqref{hyp:z}.
In addition, all they have the following useful consequence.

\begin{rem}\label{Rem:diffeomorphism}
The conditions \eqref{hyp:c}-\eqref{equiCA} imply that map $(\alpha,\lambda)\mapsto z_\lambda(t,\alpha)$ is a diffeomorphism from $\{c(\alpha)>0\}\times(-1,1)$ to $\Mixzone(t)$ with Jacobian $tc(\partial_{\alpha}z_\lambda\cdot\tau)>0$.
\end{rem} 

In Section \ref{sec:c} we will construct a suitable smooth growth-rate $c$. Once $c$ is fixed, we will still assume \eqref{hyp:z}-\eqref{equiCA} in Section \ref{sec:z}. Finally, we will construct a time-dependent pseudo-interface $z$ satisfying such conditions in Section \ref{sec:existence}.\\

We conclude this subsection by proving an auxiliary lemma which allows to integrate by parts, under certain conditions, on the domain with cusp singularities $\Mixzone$.

\begin{lemma}\label{lemma:parts}
Fix $0\leq t\leq T$. Let $f\in L^\infty(\R^2)$ satisfying that $f\in C^1$ with $\nabla\cdot f=0$ outside $\partial\Omega_{+}(t)\cup\partial\Omega_{-}(t)$ and with well-defined continuous limits
\begin{equation}\label{traces}
\begin{split}
f_{a}^{a}(\alpha)&:=\lim_{\Omega_a(t)\ni x\rightarrow z_a(t,\alpha)}f(x),\\
f_{a}^{\mathrm{mix}}(\alpha)&:=\lim_{\Mixzone(t)\ni x\rightarrow z_a(t,\alpha)}f(x),\quad\quad tc(\alpha)>0,
\end{split}
\end{equation}
whenever $z_a(t,\alpha)\in\partial\Omega_{r}(t)$ for $a=\pm$ and $r=+,-,\mathrm{mix}$. Then, for every $\phi\in C_c^1(\R^2)$,
\begin{align*}
\int_{\R^2}f\cdot\nabla\phi\dif x
&=\int_{tc(\alpha)=0}(f_+^+-f_-^-)\cdot\partial_{\alpha}z^\perp(\phi\circ z)\dif\alpha\\
&+\sum_{a=\pm}a\int_{tc(\alpha)>0}(f_a^a-f_a^{\mathrm{mix}})\cdot\partial_{\alpha}z_a^\perp(\phi\circ z_a)\dif\alpha.
\end{align*}
\end{lemma}
\begin{proof}
First of all we split the integral over $\R^2$ into $\Omega_{+}(t)$, $\Omega_{-}(t)$ and $\Mixzone(t)$.
On the one hand, by applying the Gauss divergence theorem on the regular domains $\Omega_{a}(t)$ for $a=\pm$, we get
$$\int_{\Omega_a(t)}f\cdot\nabla\phi\dif x
=a\int_{\T}f_a^a\cdot\partial_{\alpha}z_a^\perp(\phi\circ z_a)\dif\alpha,$$
where we have applied that
$\nabla\cdot f=0$ outside $\partial\Omega_{+}(t)\cup \partial\Omega_{-}(t)$
and that the normal vector to $\partial\Omega_{a}(t)$ pointing outward is $a\partial_{\alpha}z_a(t)^\perp$. This concludes the proof for $t=0$.
Now we pay special attention to the cusp singularities in $\Mixzone(t)$ for $0<t\leq T$.
For any $\varepsilon>0$ we define
$$\Mixzone^\varepsilon(t):=\{z_{\lambda}(t,\alpha)\,:\,c(\alpha)>\varepsilon,\,\lambda\in(-1,1)\},$$
which forms an exhaustion by Lipschitz domains of $\Mixzone(t)$. Hence, we can apply first the dominated convergence and then
the Gauss divergence theorem on $\Mixzone^\varepsilon(t)$ to obtain
$$
\int_{\Mixzone(t)}f\cdot \nabla \phi\dif x =\lim_{\varepsilon\to 0} \int_{\Mixzone^\varepsilon(t)}f\cdot \nabla \phi\dif x=\lim_{\varepsilon\to 0}\int_{\partial\Mixzone^\varepsilon(t)}(f^{\text{mix}}\cdot n^\varepsilon)\phi\dif\sigma,
$$
where $f^{\text{mix}}$ denotes the limit of $f$ on $\partial \Mixzone^\varepsilon(t)$ and $n^\varepsilon$ is the unit normal vector to $\partial \Mixzone^\varepsilon(t)$ pointing outward.
Since $f\in L^\infty$ it follows that the contribution of the boundary integral on $c(\alpha)=\varepsilon$, $\lambda\in(-1,1)$ is zero in the limit $\varepsilon\to 0$. Therefore, we deduce that
$$\int_{\Mixzone(t)}f\cdot \nabla \phi\dif x
=-\sum_{a=\pm}a\int_{tc(\alpha)>0}f_a^{\mathrm{mix}}\cdot\partial_{\alpha}z_a^\perp(\phi\circ z_a)\dif\alpha.$$
This concludes the proof.
\end{proof}

\subsection{The Subsolution}\label{sec:subsol}
\subsubsection{The density}\label{sec:density}
Following \cite{ForsterSzekelyhidi18,NoisetteSzekelyhidi20} we declare $\bar{\rho}=0$ on $\Mixzone$:
\begin{equation}\label{density}
\bar{\rho}(t,x):=\car{\Omega_{+}(t)}(x)-\car{\Omega_{-}(t)}(x),
\end{equation}
with $\Omega_{\pm}(t)$ given in \eqref{Omegat}.
As a result, $\partial_1\bar{\rho}(t)$ is a Dirac measure supported on $\partial\Omega_{+}(t)\cup\partial\Omega_{-}(t)$ with density $(\partial_{\alpha}z_a(t))_2$ on each $\partial\Omega_a(t)$ for $a=\pm$.

\begin{lemma}\label{rho:Dirac}
For every $\phi\in C_c^1(\R^2)$ and $0\leq t\leq T$,
$$\int_{\R^2}\bar{\rho}(t)\partial_1\phi\dif x
=-\sum_{a=\pm}\int_{\T}(\partial_{\alpha}z_a(t,\alpha))_2\phi(z_a(t,\alpha))\dif\alpha.$$
\end{lemma}
\begin{proof}
It follows from Lemma \ref{lemma:parts} applied to $f=\bar{\rho}$ because, in this case, we have $1\cdot\partial_{\alpha}z_a^\perp=-(\partial_{\alpha}z_a)_2$, $\bar{\rho}_a^a=a$ and $\bar{\rho}_a^{\mathrm{mix}}=0$.
\end{proof}

\begin{figure}[h!]
	\centering
\includegraphics[width=\textwidth]{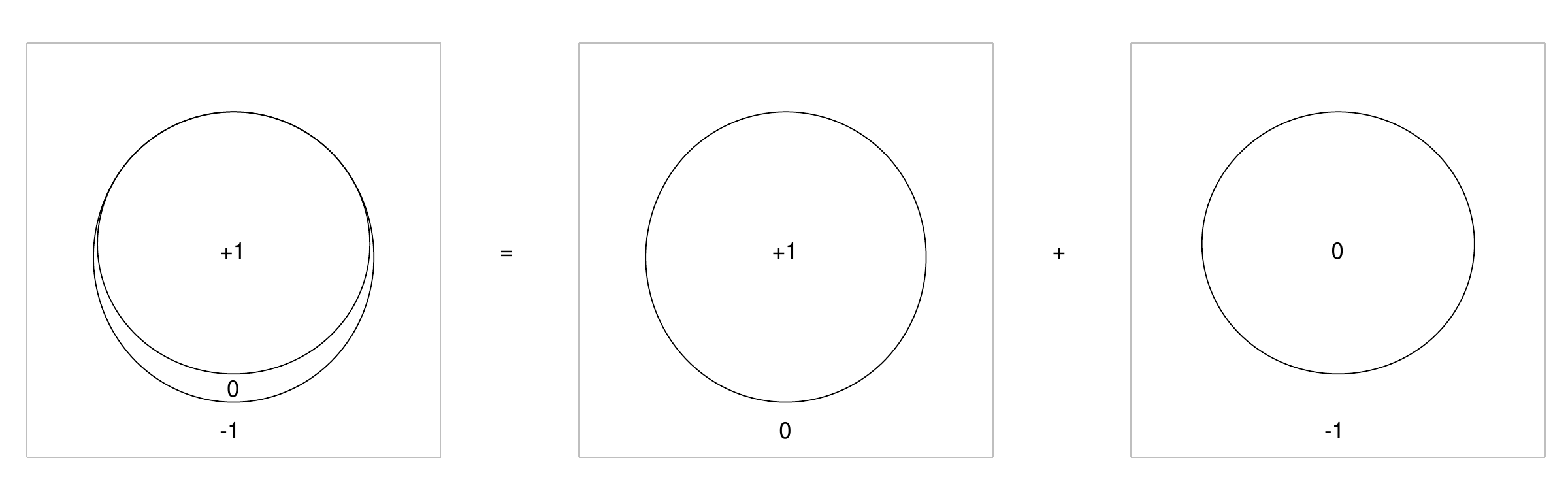}
	\caption{The macroscopic density $\bar{\rho}$ \eqref{density} can be decomposed as the sum of the contribution of $\rho_+$ on $\Omega_{+}\cup\Mixzone$ and $\rho_-$ on $\Omega_{-}\cup\Mixzone$. In this way, the cusp singularities in $\Mixzone$ can be understood as the superposition of regular domains.}
	\label{fig:cusp}
\end{figure}

\subsubsection{The velocity}\label{sec:velocity}

In view of Lemma \ref{rho:Dirac} we define $\bar{v}$ by means of the Biot-Savart law (cf.~\eqref{v:stable})
\begin{equation}\label{velocity}
\bar{v}(t,x)^*=-\frac{1}{2\pi i}\sum_{b=\pm}\int_{\T}\frac{(\partial_{\alpha}z_{b}(t,\beta))_2}{x-z_{b}(t,\beta)}\dif\beta,\quad\quad x\neq z_b(t,\beta).
\end{equation}
Observe that $\bar{v}$ is continuous (indeed $C_t^1C^\omega$) outside $\partial\Omega_+\cup\partial\Omega_-$.
 Proposition \ref{prop:v}  shows hat $\bar{v}$ satisfies the equations \eqref{RIPM:2}\eqref{RIPM:3}, the boundedness condition \eqref{velocity:bdd} and that has well-defined continuous limits \eqref{traces}. Proposition \ref{prop:vnormal} shows that the normal component of $\bar{v}$ on $\partial\Omega_+\cup\partial\Omega_-$ is well-defined and continuous. 
 Figure \ref{fig:cusp} explains why this is not surprising.

\begin{proposition}\label{prop:v}
Let $\bar{\rho}$ be as in \eqref{density}. The unique velocity satisfying \eqref{RIPM:2}\eqref{RIPM:3} which additionally vanishes as $|x|\rightarrow\infty$ is precisely \eqref{velocity}.
Moreover, $\bar{v}$ is uniformly bounded on $[0,T]\times\R^2$ and has well-defined continuous limits \eqref{traces}.
\end{proposition}

\begin{proof}
\textit{Step 1.~$\bar{v}$ in \eqref{velocity} is uniformly bounded and has well-defined continuous limits \eqref{traces}}.
First of all notice that $\bar{v}$ is continuous (indeed $C_t^1C^\omega$) outside $\partial\Omega_{+}\cup\partial\Omega_{-}$. Moreover, for any $(t,x)\notin\partial\Omega_{+}\cup\partial\Omega_{-}$ it holds that
$$|\bar{v}(t,x)|
=\frac{1}{2\pi}\left|\sum_{b=\pm}\int_{\T}\left(\frac{1}{x-z_{b}(t,\beta)}-\frac{1}{x-z_b(t,0)}\right)(\partial_{\alpha}z_{b}(t,\beta))_2\dif\beta\right|
\leq\frac{\ell_\circ^2}{8\pi}\sum_{b=\pm}\frac{\|\partial_{\alpha}z_b\|_{C_tC^0}^2}{\mathrm{dist}(x,\partial\Omega_b(t))^2},$$
that is, $\bar{v}$ decays as $|x|^{-2}$ when $|x|\rightarrow\infty$.

Let us manipulate the expression \eqref{velocity} to help better understand the behavior of $\bar{v}$ near the boundary $\partial\Omega_+\cup\partial\Omega_-$.
We will use de index w.r.t.~$z_b(t)$ of points $x$ outside $\partial\Omega_+(t)\cup\partial\Omega_-(t)$:
$$\mathrm{Ind}_{z_b(t)}(x)
:=\frac{1}{2\pi i}\int_{z_b(t)}\frac{\dif z}{x-z}
=\frac{1}{2\pi i}\int_{\T}\frac{\partial_{\alpha}z_b(t,\beta)}{x-z_b(t,\beta)}\dif\beta.$$
Recall that $z(t)$ is clockwise oriented ($\circlearrowright$).
Hence, the Cauchy's argument principle yields
\begin{equation}\label{Ind}
\begin{split}
\mathrm{Ind}_{z_+(t)}(x)
&=\car{\Omega_{+}(t)}(x),\\
\mathrm{Ind}_{z_-(t)}(x)
&=1-\car{\Omega_{-}(t)}(x).
\end{split}
\end{equation}
In order to compute the limits \eqref{traces},
for any $x$ outside $\partial\Omega_+(t)\cup\partial\Omega_-(t)$ but close enough and $a=\pm$, we take 
$\alpha(x,a)\in\T$ minimizing $|x-z_a(t,\alpha)|$.
Then,   
a straightforward identity of complex numbers yields
\begin{align}
\bar{v}(t,x)^*
&=\sum_{b=\pm}\left(\frac{1}{2\pi i}\int_{\T}\frac{(\partial_{\alpha}z_{b}(t,\beta))_2}{x-z_{b}(t,\beta)}\left(\frac{\partial_{\alpha}z_b(t,\beta)}{(\partial_{\alpha}z_b(t,\beta))_2}\frac{(\partial_{\alpha}z_b(t,\alpha))_2}{\partial_{\alpha}z_b(t,\alpha)}-1\right)\dif\beta
-\frac{(\partial_{\alpha}z_b(t,\alpha))_2}{\partial_{\alpha}z_b(t,\alpha)}\mathrm{Ind}_{z_b(t)}(x)\right)\nonumber\\
&=\sum_{b=\pm}\left(\frac{1}{2\pi i}\frac{1}{\partial_{\alpha}z_{b}(t,\alpha)}
\int_{\T}\frac{\partial_{\alpha}z_{b}(t,\alpha)\cdot\partial_{\alpha}z_b(t,\beta)^\perp}{x-z_{b}(t,\beta)}\dif\beta
-\frac{(\partial_{\alpha}z_b(t,\alpha))_2}{\partial_{\alpha}z_b(t,\alpha)}\mathrm{Ind}_{z_b(t)}(x)\right).\label{vnearGamma}
\end{align}
On the one hand, the regularity conditions \eqref{hyp:c}-\eqref{equiCA} allow to apply the dominated convergence theorem on the first term in \eqref{vnearGamma} as $x\to z_a(t,\alpha)$.
On the other hand, the limits $\Omega_{r}(t)\ni x\rightarrow z_a(t,\alpha)$ in the second term in \eqref{vnearGamma} change depending on the region $r=+,-,\text{mix}$ where $x$ is coming from due to \eqref{Ind}. Therefore,
$\bar{v}$ has well-defined continuous limits \eqref{traces}
\begin{equation}\label{v:limits}
\begin{split}
\bar{v}_+^+
&=V_+-\frac{(\partial_{\alpha}z_-)_2}{(\partial_{\alpha}z_-)^*}-\frac{(\partial_{\alpha}z_+)_2}{(\partial_{\alpha}z_+)^*},\\
\bar{v}_+^{\mathrm{mix}}
&=V_+-\frac{(\partial_{\alpha}z_-)_2}{(\partial_{\alpha}z_-)^*},\\
\bar{v}_-^{\mathrm{mix}}
&=V_--\frac{(\partial_{\alpha}z_-)_2}{(\partial_{\alpha}z_-)^*},\\
\bar{v}_-^-
&=V_-,
\end{split}
\end{equation}
where 
$$
V_{a}:=\sum_{b=\pm}V_{a,b}
\quad\textrm{with}\quad
V_{a,b}(t,\alpha)^*
:=\frac{1}{2\pi i}\frac{1}{\partial_{\alpha}z_{b}(t,\alpha)}
\int_{\T}\frac{\partial_{\alpha}z_{b}(t,\alpha)\cdot\partial_{\alpha}z_b(t,\beta)^\perp}{z_{a}(t,\alpha)-z_{b}(t,\beta)}\dif\beta,$$
for $(t,\alpha)\in[0,T]\times\T$
and $a,b=\pm$.
Finally, it follows that
$\bar{v}$ is uniformly bounded on $[0,T]\times\R^2$.

\textit{Step 2.~$\bar{v}$ satisfies \eqref{RIPM:2}\eqref{RIPM:3}}. Observe that $\bar{v}(t)^*$ is holomorphic outside $\partial\Omega_+(t)\cup\partial\Omega_-(t)$ for all $0\leq t\leq T$. Thus, the Cauchy-Riemann equations imply that $\nabla\cdot\bar{v}(t)=\nabla^\perp\cdot\bar{v}(t)=0$ outside $\partial\Omega_+(t)\cup\partial\Omega_-(t)$.
Notice also that $z_{+}=z_{-}$ and $V_{+}=V_{-}$ on $tc(\alpha)=0$. In particular,
\begin{align*}
\bar{v}_+^+-\bar{v}_-^-
&=-\left(\frac{(\partial_{\alpha}z_+)_2}{(\partial_{\alpha}z_+)^*}+\frac{(\partial_{\alpha}z_-)_2}{(\partial_{\alpha}z_-)^*}\right)\quad\textrm{on}\quad tc(\alpha)=0,\\
\bar{v}_a^a-\bar{v}_a^\mathrm{mix}
&=-a\frac{(\partial_{\alpha}z_a)_2}{(\partial_{\alpha}z_a)^*}\hspace{2.25cm}\quad\textrm{on}\quad tc(\alpha)>0.
\end{align*}
Let $\phi\in C_c^1(\R^2)$ and $0<t\leq T$.
Then, by applying Lemma \ref{lemma:parts} to $f=\bar{v}$ and $\bar{v}^\perp$, we deduce that
\begin{align*}
\int_{\R^2}\bar{v}\cdot\nabla\phi\dif x
&=-\sum_{a=\pm}\int_{\T}\frac{(\partial_{\alpha}z_a)_2}{(\partial_{\alpha}z_a)^*}\cdot\partial_{\alpha}z_a^\perp(\phi\circ z_a)\dif\alpha=0,\\
\int_{\R^2}\bar{v}\cdot\nabla^\perp\phi\dif x
&=\sum_{a=\pm}\int_{\T}\frac{(\partial_{\alpha}z_a)_2}{(\partial_{\alpha}z_a)^*}\cdot\partial_{\alpha}z_a(\phi\circ z_a)\dif\alpha
=\sum_{a=\pm}\int_{\T}(\partial_{\alpha}z_a)_2(\phi\circ z_a)\dif\alpha.
\end{align*}
These identities jointly with Lemma \ref{rho:Dirac} imply that $\bar{v}$ satisfies \eqref{RIPM:2}\eqref{RIPM:3}.\\
\indent\textit{Step 3.~Uniqueness}. Finally, it is easy to check that any solution to \eqref{RIPM:2}\eqref{RIPM:3} has the form $\bar{u}=\bar{v}+f^*$ for some (time-dependent) entire function $f$. Thus, if $\bar{u}$ vanishes as $|x|\rightarrow\infty$ too, the Liouville's theorem implies that $f=0$.
\end{proof}

In the next lemma we deal with the normal component of $\bar{v}$ at the boundary of the mixing zone $\partial\Omega_+(t)\cup\partial\Omega_-(t)$. We will use the notation
from Lemma~\ref{lemma:parts} for the outer an inner limits and the operators $B,B_a$ defined in the intro \eqref{B:unstable}, \eqref{Bab:unstable}.

\begin{proposition}\label{prop:vnormal} Let $a=\pm$ and $r=+,-\mathrm{mix}$.
Then, it holds that
$$(\bar{v}_a^r-B_a)\cdot\partial_{\alpha}z_a^\perp=0,$$
on $[0,T]\times\T$. In particular,
$$(\bar{v}_a^r-B)\cdot\partial_{\alpha}z^\perp
=0
\quad\textrm{on}\quad
tc(\alpha)=0.
$$
\end{proposition}

\begin{proof}
Let $x$ be outside $\partial\Omega_+(t)\cup\partial\Omega_-(t)$ but close enough.
Firstly, using that $(\mathrm{Ind}_{z_b(t)}(x))_2=0$, it follows that
the velocity \eqref{velocity} can be written as (cf.~\eqref{v:stable}\eqref{CAP:stable})
$$
\bar{v}(t,x)
=-\frac{1}{2\pi }\sum_{b=\pm}\int_{\T}\left(\frac{1}{x-z_{b}(t,\beta)}\right)_1\partial_{\alpha}z_{b}(t,\beta)\dif\beta.
$$
In particular,
$$\bar{v}(t,x)\cdot\partial_{\alpha}z_a(t,\alpha)^\perp
=
\frac{1}{2\pi }\sum_{b=\pm}\int_{\T}\left(\frac{1}{x-z_{b}(t,\beta)}\right)_1(\partial_{\alpha}z_{a}(t,\alpha)-\partial_{\alpha}z_{b}(t,\beta))\dif\beta\cdot\partial_{\alpha}z_a(t,\alpha)^\perp,$$
where we take $\alpha\in\T$ as in \eqref{vnearGamma}.
Thus, it remains to show that we can take the limit $x\to z_a(t,\alpha)$ in the r.h.s.~above.
By writing $z_a=z_b-i(a-b)tc\tau$, we split the above integrals into
$$\int_{\T}\left(\frac{1}{x-z_{b}(t,\beta)}\right)_1(\partial_{\alpha}z_{b}(t,\alpha)-\partial_{\alpha}z_{b}(t,\beta))\dif\beta
-i(a-b)t\partial_{\alpha}(c(\alpha)\tau(\alpha))\left(\int_{\T}\frac{\dif\beta}{x-z_{b}(t,\beta)}\right)_1.$$
Analogously to \eqref{vnearGamma}, the regularity conditions \eqref{hyp:c}-\eqref{equiCA} allow to apply the dominated convergence theorem on the first term as $x\to z_a(t,\alpha)$.
Notice that the second term vanishes for $(a-b)tc(\alpha)=0$. Otherwise, we can consider directly $x\to z_a(t,\alpha)$.
This implies the first statement. As a by-product, we have seen that $B_{a,b}$ is split into
\begin{equation}\label{Bab:split}
\begin{split}
B_{a,b}(t,\alpha)&=\frac{1}{2\pi}\int_{\T}\left(\frac{1}{z_{a}(t,\alpha)-z_{b}(t,\beta)}\right)_1(\partial_{\alpha}z_{b}(t,\alpha)-\partial_{\alpha}z_{b}(t,\beta))\dif\beta\\
&-i(a-b)t\partial_{\alpha}(c(\alpha)\tau(\alpha))\frac{1}{2\pi}\left(\int_{\T}\frac{\dif\beta}{z_{a}(t,\alpha)-z_{b}(t,\beta)}\right)_1.
\end{split}
\end{equation}
Finally, the second statement follows from the fact that $z_+=z_-$ and $B_+=B_-=B$ on $tc(\alpha)=0$.
\end{proof}

\subsubsection{The relaxed momentum}\label{sec:m}
In view of the inequality \eqref{hull}, it seems suitable to define $\bar{m}$ as
\begin{equation}\label{m}
\bar{m}:=\bar{\rho}\bar{v}-(1-\bar{\rho}^2)(\gamma+\tfrac{1}{2}i),
\end{equation}
in terms of some $\gamma\in C(\Mixzone;\R^2)$ to be determined (\cite{Szekelyhidi12,CCFpp,ForsterSzekelyhidi18, NoisetteSzekelyhidi20}). Moreover, for any $0<t\leq T$ we assume that $\gamma(t)\in C^1(\Mixzone(t))$ with continuous limits
$$\gamma_{a}(t,\alpha):=\lim_{\Mixzone(t)\ni x\rightarrow z_a(t,\alpha)}\gamma(t,x),$$
whenever $z_a(t,\alpha)\in\partial\Mixzone(t)$ for $a=\pm$ and $c(\alpha)>0$.

\subsection{Compatibility between the mixing zone and the subsolution}\label{sec:conditionscz}
In the next proposition we derive the conditions for $(c,z,\gamma)$ under which the corresponding $(\bar{\rho},\bar{v},\bar{m})$ given in \eqref{density},\eqref{velocity} and\eqref{m} becomes a subsolution.

\begin{proposition}\label{conditions:z,gamma}
Assume that $(c,z)$ satisfies \eqref{hyp:c}-\eqref{equiCA} for some $T>0$.	
The triplet $(\bar{\rho},\bar{v},\bar{m})$ given in \eqref{density},\eqref{velocity} and \eqref{m} defines a subsolution to IPM if and only if the triplet $(c,z,\gamma)$ satisfies the following equations on $\partial\Omega_a$ for $a=\pm$
\begin{subequations}\label{bcond}
\begin{align}
(\partial_tz-B)\cdot\partial_{\alpha}z^\perp&=0\quad\textrm{on}\quad tc(\alpha)=0,\label{bcond:1}\\
(\partial_tz_{a}-B_{a}-a(\gamma_{a}+\tfrac{1}{2}i))\cdot\partial_{\alpha}z_{a}^\perp&=0\quad\textrm{on}\quad tc(\alpha)>0,\label{bcond:2}
\end{align}
\end{subequations}
 and the following conditions on $\Mixzone$
\begin{subequations}\label{gamma}
\begin{align}
\nabla\cdot\gamma&=0,\label{divgamma}\\
|\gamma|&<\tfrac{1}{2}.\label{hullcond}
\end{align}
\end{subequations}
\end{proposition}
\begin{proof}
Recall that $(\bar{\rho},\bar{v},\bar{m})$ already satisfies the equations \eqref{RIPM:2}\eqref{RIPM:3} and the conditions \eqref{hull:1} \eqref{velocity:bdd} (see Prop.~\ref{prop:v}). Moreover, it is clear that
$(\bar{\rho},\bar{v},\bar{m})$ satisfies the equation \eqref{RIPM:1} on $\Mixzone$ and the inequality \eqref{hull:2} if and only if \eqref{gamma} holds. Thus, it remains to analyze \eqref{RIPM:1} outside $\Mixzone$.\\
\indent Let $\phi\in C_c^1(\R^3)$ and $0< t\leq T$.
On the one hand, since $\bar{\rho}=0$ on $\Mixzone$ and $\bar{\rho}=\pm 1$ on the regular domains $\Omega_{\pm}(t)$, an integration by parts yields
\begin{align*}
&\int_{0}^t\int_{\R^2}\bar{\rho}\pa_t \phi\dif x\dif s-\int_{\R^2}\bar{\rho}(t)\phi(t)\dif x+\int_{\R^2}\rho^\circ\phi^\circ\dif x
\\&= -2\,\int_0^t \int_{c(\alpha)=0}\partial_tz\cdot\partial_{\alpha}z^\perp(\phi\circ Z)\dif\alpha\dif s
-\sum_{a=\pm}\,\int_0^t \int_{c(\alpha)>0}\partial_tz_{a}\cdot\partial_{\alpha}z_{a}^\perp(\phi\circ Z_a)\dif\alpha\dif s,
\end{align*}
where $Z(t,\alpha):=(t,z(t,\alpha))$ and $Z_a(t,\alpha):=(t,z_a(t,\alpha))$.
On the other hand, notice \eqref{m} reads as
$$\bar{m}(t,x)
=\left\lbrace\begin{array}{rl}
\pm\bar{v}(t,x),& x\in\Omega_{\pm}(t),\\[0.1cm]
-(\gamma(t,x)+\tfrac{1}{2}i), & x\in\Mixzone(t).
\end{array}
\right.$$
In particular, $\bar{m}$ satisfies the assumptions in Lemma \ref{lemma:parts}. Therefore, it follows that
\begin{align*}
\int_{\R^2}\bar{m}\cdot\nabla\phi\dif x
&=2\int_{c(\alpha)=0}B\cdot\partial_{\alpha}z^\perp(\phi\circ Z)\dif\alpha\\
&+\sum_{a=\pm}\int_{c(\alpha)>0}B_a\cdot\partial_{\alpha}z_a^\perp(\phi\circ Z_a)\dif\alpha\\
&+\sum_{a=\pm}a\int_{c(\alpha)>0}(\gamma_a+\tfrac{1}{2}i)\cdot\partial_{\alpha}z_a^\perp(\phi\circ Z_a)\dif\alpha,
\end{align*}
where we have applied Proposition \ref{prop:vnormal} in the first two lines above. 
In summary, we have seen that
\begin{align*}
&\int_{0}^t\int_{\R^2}(\bar{\rho}\pa_t \phi+\bar{m}\cdot\nabla\phi)\dif x\dif s-\int_{\R^2}\bar{\rho}(t)\phi(t)\dif x+\int_{\R^2}\rho^\circ\phi^\circ\dif x\\
&=-2\int_0^t \int_{c(\alpha)=0}(\partial_tz-B)\cdot\partial_{\alpha}z^\perp(\phi\circ Z)\dif\alpha\dif s\\
&-\sum_{a=\pm}\int_0^t \int_{c(\alpha)>0}((\partial_tz_{a}-B_a-a(\gamma_{a}+\tfrac{1}{2}i))\cdot\partial_{\alpha}z_{a}^\perp)(\phi\circ Z_a)\dif\alpha\dif s.
\end{align*}
This concludes the proof.
\end{proof}

We conclude this section by showing that we can construct $\gamma(t,x)$ satisfying the requirements in Proposition \ref{conditions:z,gamma} provided that $(c,z)$ satisfies certain conditions. 
Observe that  $\{c(\alpha)>0\}$ is open and thus a (countable) union of disjoint intervals $(\alpha_1,\alpha_2)$. Recall the definition of $B$ and $D$ from \eqref{B:unstable} and \eqref{def:D}.

\begin{lemma}\label{conditions:z} Assume that $(c,z)$ satisfies \eqref{hyp:c}-\eqref{equiCA} for some $T>0$.  Assume further that the following conditions hold uniformly on $\{c(\alpha)>0\}$
\begin{equation}\label{c:hull}
|2c(\alpha)+\partial_{\alpha}z^\circ_1(\alpha)|<1,
\end{equation}
and
\begin{subequations}\label{cond:z}
\begin{align}
\partial_t z-B_a&=o(1),\label{cond:z:1}\\
\frac{1}{tc(\alpha)}\int_{\alpha_1}^{\alpha}((\partial_tz-B)\cdot\partial_{\alpha}z^\perp+tD\cdot\partial_{\alpha}(c\tau))\dif\alpha'&=o(1),\label{cond:z:2}
\end{align}
\end{subequations}
as $t\rightarrow 0$, for $a=\pm$ and $\alpha\in (\alpha_1,\alpha_2)$ connected component of $\{c(\alpha)>0\}$.
Then, there exists $0<T'\leq T$ and $\gamma(t,\alpha)$
satisfying \eqref{bcond:2}-\eqref{gamma}  as long as $0<t\leq T'$.
\end{lemma}

\begin{proof}
\textit{Step 1.~Analysis of \eqref{bcond:2}-\eqref{gamma}}.
For simplicity we may assume w.l.o.g.~that there is one connected component $(\alpha_1,\alpha_2)=\{c(\alpha)>0\}$.
Recall that $(\alpha,\lambda)\mapsto z_{\lambda}(t,\alpha)$ is a diffeomorphism from $(\alpha_1,\alpha_2)\times(-1,1)$ to $\Mixzone(t)$ (cf.~Rem.~\ref{Rem:diffeomorphism}).
In particular, since $\Mixzone(t)$ is simply-connected, \eqref{divgamma} implies that $\gamma(t)=\nabla^\perp g(t)$ for some $g(t)\in C^1(\Mixzone(t))$ to be determined. Moreover, $g$ can be defined in terms of some $G$ in $(\alpha,\lambda)$-coordinates as $$g(Z(t,\alpha,\lambda)):=G(t,\alpha,\lambda),$$
where $Z(t,\alpha,\lambda):=(t,z_\lambda(t,\alpha))$.
Notice that (recall $z_\lambda=z-\lambda tc\tau^\perp$)
$$\partial_\alpha G=(\nabla g \circ Z) \cdot\partial_{\alpha}z_{\lambda},
\quad\quad
\partial_\lambda G=-tc(\nabla g \circ Z)\cdot\tau^\perp.$$
On the one hand, the boundary conditions \eqref{bcond:2} for $\gamma$ read as
\begin{equation}\label{G:pm}
\partial_{\alpha}G(t,\alpha,a)
=(a(\partial_tz-B_{a})-i(c\tau+\tfrac{1}{2}))\cdot\partial_{\alpha}z_{a}^\perp,
\quad\quad a=\pm.
\end{equation}
On the other hand, for \eqref{hullcond} notice that
\begin{equation}\label{nablag}
\nabla g \circ Z
=\frac{1}{\partial_{\alpha}z_{\lambda}\cdot\tau}\left(\partial_{\alpha}G\tau-\frac{1}{tc}\partial_{\lambda}G\partial_{\alpha}z_{\lambda}^\perp\right).
\end{equation}
In particular,
$$\partial_{\alpha}z_\lambda\cdot\tau
=\partial_{\alpha}z^\circ\cdot\tau+t\left(\dashint_0^t\partial_tz\dif s-\lambda\partial_{\alpha}(c\tau^\perp)\right)\cdot\tau=\partial_{\alpha}z^\circ\cdot\tau+O(t),$$
with $\partial_{\alpha}z^\circ\cdot\tau> 0$ uniformly on $c(\alpha)>0$ by \eqref{Angle}.
Therefore, assuming that $c$ satisfies \eqref{c:hull}, it is enough to find $G$ satisfying \eqref{G:pm} and the following growth conditions
\begin{equation}\label{G:o}
\partial_{\alpha}G=o(1)-(c\tau+\tfrac{1}{2})\cdot\partial_{\alpha}z_\lambda,
\quad\quad
\frac{1}{tc}\partial_{\lambda}G=o(1),
\end{equation}
uniformly on $c(\alpha)>0$
as $t\to 0$, because in this case (recall $\tau=\partial_{\alpha}z^\circ$ with $|\partial_{\alpha}z^\circ|=1$)
\begin{equation}\label{|gamma|}
|\gamma|
=|\nabla g|
\leq\left|c+\tfrac{1}{2}\tfrac{\partial_{\alpha}z_1^\circ}{\partial_{\alpha}z^\circ\cdot\tau}\right|+o(1)<\tfrac{1}{2}.
\end{equation}
\indent\textit{Step 2.~Ansatz for $G$}.
We declare
\begin{equation}\label{G}
G(t,\alpha,\lambda)
:=\int_{\alpha_1}^{\alpha}\left(\sum_{a=\pm}\frac{\lambda+a}{2}(\partial_tz-B_a)\cdot\partial_{\alpha}z_a^\perp-(c\tau+\tfrac{1}{2})\cdot\partial_{\alpha}z_{\lambda}\right)\dif\alpha'.
\end{equation}
Hence, it follows that
\begin{align*}
\partial_{\alpha}G
&=\sum_{a=\pm}\frac{\lambda+a}{2}(\partial_tz-B_a)\cdot\partial_{\alpha}z_a^{\perp}
-(c\tau+\tfrac{1}{2})\cdot\partial_{\alpha}z_\lambda,\\
\partial_{\lambda}G
&=\int_{\alpha_1}^{\alpha}\left((\partial_tz-B)\cdot\partial_{\alpha}z^\perp+tD\cdot\partial_{\alpha}(c\tau)\right)\dif\alpha'.
\end{align*}
Notice that \eqref{G:pm} is satisfied.
Finally, assuming that $z$ satisfies \eqref{cond:z}, then \eqref{G:o} holds.
\end{proof}

\begin{rem}\label{rem:ctau}
Following the previous proof, notice that for  $z(t,\alpha)=(\alpha,f(t,\alpha))$ and $\tau=(-1,0)$ as in \cite{ForsterSzekelyhidi18,NoisetteSzekelyhidi20}, the admissible regime for $c(\alpha)$	reads as
$$|2c(\alpha)-1|<1,$$
which clear is incompatibly with $c(\alpha)=0$.
Remarkably, the authors in \cite{NoisetteSzekelyhidi20} achieved that $c(\alpha)\to 0$ in the limiting case $|\alpha|\to\infty$ for $\alpha\in\R$.
\end{rem}

In view of Proposition \ref{conditions:z,gamma} and Lemma \ref{conditions:z}, we need to find a growth-rate $c(\alpha)$ with certain regularity \eqref{hyp:c} and satisfying the inequality \eqref{c:hull}, and a time-dependent pseudo-interface $z(t,\alpha)$ satisfying the regularity assumptions \eqref{hyp:z}-\eqref{equiCA} and the relations \eqref{bcond:1}\eqref{cond:z}.

\section{The growth-rate}\label{sec:c}
In this section we declare a suitable growth-rate $c$, and also a partition of the unity $\{\psi_0,\psi_1\}$ as discussed in the intro, in terms of $z^\circ\in C^{1,\delta}(\T;\R^2)$. Let us recall what we need.
On the one hand, 
we must construct a regular enough $c$ \eqref{hyp:c} satisfying
the inequality \eqref{c:hull} uniformly on $\{c(\alpha)>0\}$. At the same time, for the relation \eqref{cond:z:2} it is convenient to control the monotonicity of $c$ near the boundary of $\{c(\alpha)>0\}$. On the other hand, we shall construct $\{\psi_0,\psi_1\}$ subordinated to  $\{\partial_{\alpha}z_1^\circ(\alpha)>0\}$ and $\{c(\alpha)>0\}$ respectively, and satisfying $\partial_{\alpha}z_1^\circ> 0$ uniformly on $\mathrm{supp}\,\psi_0$.

In view of Figure \ref{fig:c}, it seems clear that we have enough flexibility to construct such functions. The aim of this section is to make it quantitatively.

\begin{figure}[h!]
	\centering
	\includegraphics[width=0.5\textwidth]{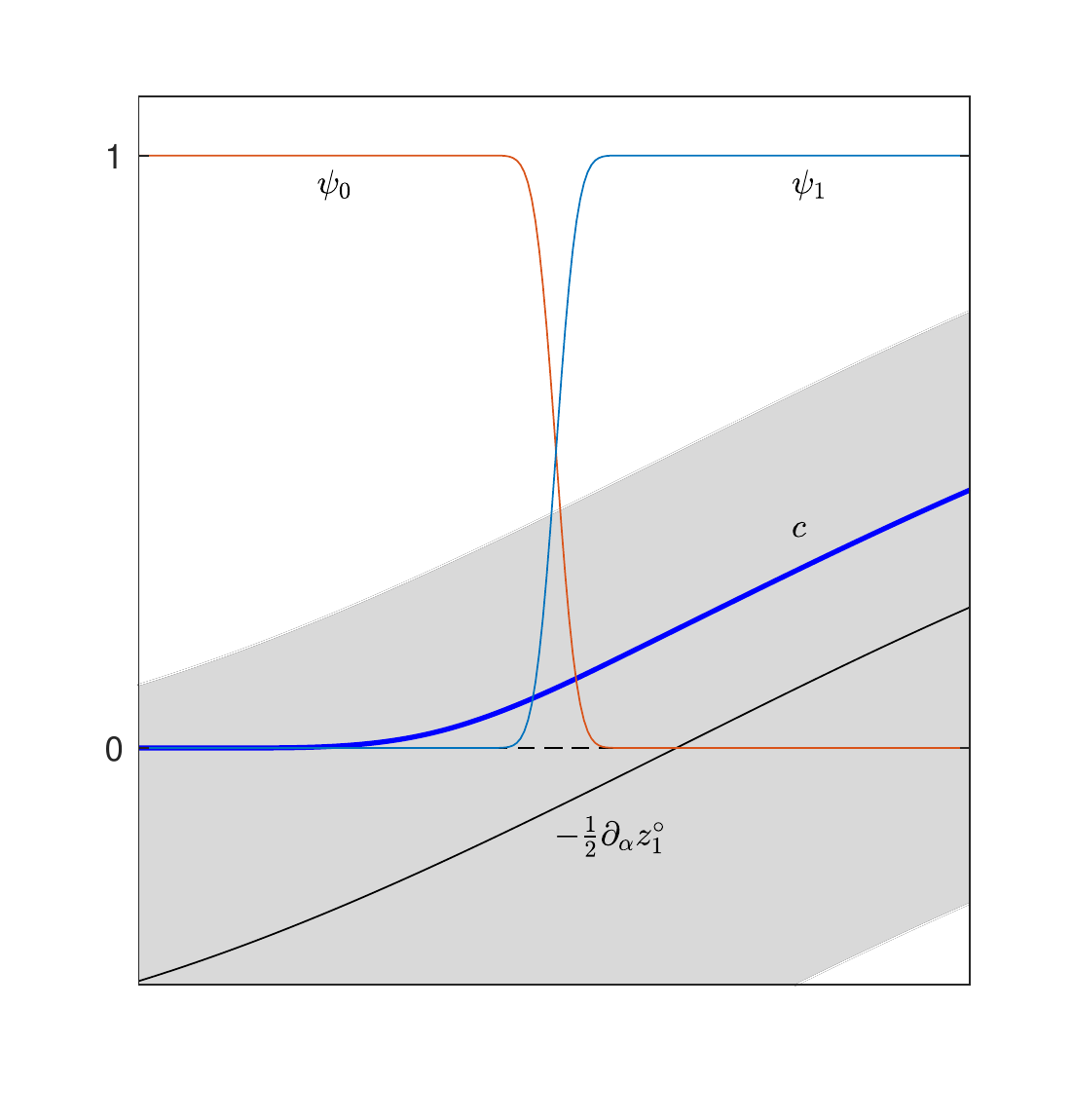}
	\caption{A regular growth-rate $c(\alpha)\geq 0$ (blue) satisfying the inequality $|2c(\alpha)+\partial_{\alpha} z_1^\circ(\alpha)|<1$ (gray zone) uniformly on $\{c(\alpha)>0\}$.
	A partition of the unity $\{\psi_0,\psi_1\}$ subordinated to  $\{\partial_{\alpha}z_1^\circ(\alpha)>0\}$ and $\{c(\alpha)>0\}$ respectively, and satisfying $\partial_{\alpha}z_1^\circ> 0$ uniformly on $\mathrm{supp}\,\psi_0$.}
	\label{fig:c}
\end{figure}

Let us introduce the following sets $I_\eta$ which will be very useful for these purposes.

\begin{lemma}\label{lemma:Isigma} Given $-1\leq\eta\leq 1$ we denote
\begin{equation}\label{sigma}
I_{\eta}:=\{\alpha\in\T\,:\,\partial_{\alpha}z^\circ_1(\alpha)<\eta\}.
\end{equation}
This forms an ascending chain of open subsets of $\T$. Furthermore, for any $-1<\eta_1<\eta_2<1$, we have $I_{\eta_1}\subset\subset I_{\eta_2}$ with
$$\mathrm{dist}(\partial I_{\eta_1},\partial I_{\eta_2})\geq\left(\frac{\eta_2-\eta_1}{|\partial_{\alpha}z_1^\circ|_{C^\delta}}\right)^{1/\delta}.$$
\end{lemma}
\begin{proof} First of all notice that $\partial_{\alpha}z_1^\circ(\T)=[-1,1]$. Hence, there is $\alpha_j\in\partial I_{\eta_j}$ for $j=1,2$, and thus
$$\eta_2-\eta_1
=\partial_{\alpha}z_1^\circ(\alpha_2)-\partial_{\alpha}z_1^\circ(\alpha_1)
\leq|\partial_{\alpha}z_1^\circ|_{C^\delta}\mathrm{dist}(\partial I_{\eta_1},\partial I_{\eta_2})^\delta,$$
as we wanted to prove.
\end{proof}

In order to interpolate between the Classical and the Mixing Muskat problem, we take a partition of the unity $\{\psi_0,\psi_1\}\subset C_c^\infty(\T;[0,1])$ as follows. Firstly, we define the indicator function
\begin{equation}\label{chi}
\chi_{\eta,s}(\alpha)
:=\left\lbrace\begin{array}{cl}
1, & \mathrm{dist}(\alpha,\T\setminus I_{\eta})>r,\\[0.1cm]
0, & \textrm{otherwise},
\end{array}\right.
\quad\textrm{with}\quad
r:=s\left(\frac{\eta}{|\partial_{\alpha}z_1^\circ|_{C^\delta}}\right)^{1/\delta},
\end{equation}
in terms of some parameters $0<\eta,s<1$ to be determined.
With $\chi_{\eta,s}$ we declare
\begin{equation}\label{partition}
\psi_1:=\phi_r*\chi_{\eta,s},
\quad\quad
\psi_0:=1-\psi_1,
\end{equation}
where $\phi_r$ is a standard mollifier, namely $\phi_r(\alpha):=\tfrac{1}{r}\phi(\tfrac{\alpha}{r})$ for some fixed $\phi\in C_c^\infty(-1,1)$ satisfying $\phi\geq 0$ and $\int\phi =1$, and $r>0$ is given in \eqref{chi}. Hence,
$$\|\partial_{\alpha}^k\psi_j\|_{L^\infty}
\leq \|\partial_{\alpha}^k\phi\|_{L^1}r^{-k},
\quad\quad k\geq 0,\,\,j=0,1.$$

\indent Among all the possible choices for $c$, we declare
\begin{equation}\label{c}
c:=\tfrac{1}{2}\phi_{r}*(\eta\chi_{\eta,s}+(\partial_{\alpha}z_1^\circ)_-),
\end{equation}
where $(\partial_{\alpha}z_1^\circ)_-:=-\min(\partial_{\alpha}z_1^\circ,0)$.
This $c$ is smooth with
$$\|\partial_{\alpha}^kc\|_{L^\infty}
\leq\tfrac{1}{2}(\eta+1)\|\partial_{\alpha}^k\phi\|_{L^1}r^{-k},
\quad\quad k\geq 0,$$
and satisfies $c\geq\psi_1\eta/2$ with
$\mathrm{supp}\,c=\mathrm{supp}\,\psi_1\subset\bar{I}_\eta$.
In the next lemma we show that we can take $\eta$ and $s$ in such a way that the inequality \eqref{c:hull} holds.

\begin{lemma}\label{lemma:c}
The growth-rate \eqref{c} satisfies
$$|2c+\partial_{\alpha}z_1^\circ|
\leq \eta(2+s^\delta)
\quad\textrm{on}\quad c(\alpha)>0.$$
\end{lemma}
\begin{proof}	
By writing $f=f_+-f_-$ for $f=\partial_{\alpha}z_1^\circ$, we split
$$
2c+\partial_{\alpha}z_1^\circ
=\eta\psi_1
+f_+
+(\phi_{r}*f_--f_-).
$$
On the one hand, $\eta\psi_1\leq\eta$ and also
$f_+\leq \eta$ on $\bar{I}_{\eta}$ $(\supset\mathrm{supp}\,c)$ by \eqref{sigma}.
On the other hand, by \eqref{chi}
$$|(\phi_{r}*f_--f_-)(\alpha)|
=\left|\int_{-r}^{r}(f_-(\alpha-\beta)-f_-(\alpha))\phi_{r}(\beta)\dif\beta\right|
\leq|f_-|_{C^\delta}r^\delta
\leq s^\delta\eta.$$
This concludes the proof.
\end{proof}

Let us assume from now on that $s<1/3$.
In the next lemma we show that $\partial_{\alpha}z_1^\circ> 0$ uniformly on $\mathrm{supp}\,\psi_0$ $(\supset\supset\T\setminus\mathrm{supp}\,c)$, which will be crucial for the energy estimates in Section \ref{sec:existence}.

\begin{lemma}\label{lemma:stable}
It holds that
$$\partial_{\alpha}z_1^\circ\geq \eta(1-(2s)^\delta)
\quad\textrm{on}\quad
\mathrm{supp}\,\psi_0.$$
\end{lemma}
\begin{proof} Let $\alpha\in\mathrm{supp}\,\psi_0$. If $\alpha\notin I_\eta$, simply $\partial_{\alpha}z_1^\circ(\alpha)\geq\eta$ by \eqref{sigma}.
Assume now that $\alpha\in I_\eta$. Since $\psi_1\equiv 1$ on the open set $I_\eta\setminus\bar{B}_{2r}(\partial I_\eta)$ by \eqref{chi}\eqref{partition}, necessarily $\alpha\in\bar{B}_{2r}(\partial I_\eta)$. Hence, there is $\alpha_\eta\in\partial I_\eta$ satisfying $|\alpha-\alpha_\eta|\leq 2r$, and thus
$$\partial_{\alpha}z_1^\circ(\alpha)
=\underbrace{\partial_{\alpha}z_1^\circ(\alpha_\eta)}_{=\eta}
+\underbrace{\partial_{\alpha}z_1^\circ(\alpha)-\partial_{\alpha}z_1^\circ(\alpha_\eta)}_{\geq - |\partial_{\alpha}z_1^\circ|_{C^\delta}(2r)^\delta}
\geq \eta(1-(2s)^\delta),$$
where we have applied \eqref{chi}.
\end{proof}

Next we turn to the behavior of $c$ inside $\{c(\alpha)>0\}$.
Of course $c$ is monotone in a neighborhood of $c(\alpha)=0$ and away from zero outside it,
but we give bounds for these properties that only depends on 
$s,\eta,\delta$.

\begin{lemma}\label{lemma:monotonicity}
Let $I=(\alpha_1,\alpha_2)$ be a connected component of $\{c(\alpha)>0\}$ and denote $\bar{\alpha}:=\tfrac{1}{2}(\alpha_1+\alpha_2)$.
If $|I|<2r(1/s-1)$, then $c$ is monotone on $[\alpha_1,\bar{\alpha}]$ and $[\bar{\alpha},\alpha_2]$. Otherwise, $c$ is monotone on each connected component of $I\cap B_{2r}(\partial I)$, while $c\geq\eta/2$ on $I\setminus B_{2r}(\partial I)$. Moreover $\psi_1\eta/2<c$
everywhere.
\end{lemma}
\begin{proof}
First of all observe that $\phi_{r}*(\partial_{\alpha}z_1^\circ)_-=0$ (and so $c=\psi_1\eta/2$) outside $B_r(I_0)$.\\
\indent\textit{Case $|I|<2r(1/s-1)$}.
Given $\alpha\in B_r(I_0)$, Lemma \ref{lemma:Isigma} and \eqref{chi} imply that
$$\mathrm{dist}(\alpha,\partial I)
\geq\mathrm{dist}(I_0,\partial I)-r
\geq r(1/s-1).$$
Thus, necessarily $\alpha\notin I$. Then, $I\cap B_r(I_0)=\emptyset$ and so $c=\psi_1\eta/2$ with $\psi_1$ monotone on $[\alpha_1,\bar{\alpha}]$ and $[\bar{\alpha},\alpha_2]$ by construction \eqref{chi}\eqref{partition}.\\[0.1cm]
\indent\textit{Case $|I|\geq 2r(1/s-1)$}. Given $\alpha\in B_{2r}(I_0)$ and $\beta\in B_r(\partial I)$, Lemma \ref{lemma:Isigma} and \eqref{chi} imply that
$$|\alpha-\beta|
\geq\mathrm{dist}(I_0,\partial I)-3r
\geq r(1/s-3)>0.$$
Hence, $B_{2r}(\partial I)\cap B_r(I_0)=\emptyset$ and so $c=\psi_1\eta/2$ on $I\cap B_{2r}(\partial I)$ with $\psi_1$ monotone on each connected component of $I\cap B_{2r}(\partial I)$. Finally, $c\geq\psi_1\eta/2$ with $\psi_1\equiv 1$ on $I\setminus B_{2r}(\partial I)$.
\end{proof}

Finally, Lemma \ref{lemma:monotonicity} implies the following estimates which will be useful to control \eqref{cond:z:2}.

\begin{cor}\label{cor:monotone}
In the context of Lemma \ref{lemma:monotonicity}, let us denote $$I(\alpha):=
\left\lbrace\begin{array}{cl}
[\alpha_1,\alpha], & \alpha_1\leq\alpha\leq\bar{\alpha},\\[0.1cm]
[\alpha,\alpha_2], & \bar{\alpha}<\alpha\leq\alpha_2.
\end{array}\right.$$
Then, for $k=0,1$, we have
$$\int_{I(\alpha)}|\partial_{\alpha}^kc(\alpha')|\dif\alpha'
\lesssim c(\alpha),$$
in terms of $(\eta/|\partial_{\alpha}z_1^\circ|_{C^\delta})^{1/\delta}$, $1/\eta$ and $\|\partial_{\alpha}^kc\|_{L^1}$.
\end{cor}
\begin{proof}
\textit{Case $|I|<2r(1/s-1)$}. For all $\alpha\in I$, the monotonicity of $c$ on $I(\alpha)$ implies
$$\int_{I(\alpha)}|\partial_{\alpha}^kc(\alpha')|\dif\alpha'
\leq |I(\alpha)|^{1-k}c(\alpha),$$
with
$|I(\alpha)|\leq r(1/s-1)=(1-s)(\eta/|\partial_{\alpha}z_1^\circ|_{C^\delta})^{1/\delta}\leq(\eta/|\partial_{\alpha}z_1^\circ|_{C^\delta})^{1/\delta}$.\\[0.1cm]
\indent\textit{Case $|I|\geq2r(1/s-1)$}. For all $\alpha\in I\cap B_{2r}(\partial I)$, the previous argument works with $|I(\alpha)|\leq 2r\leq (\eta/|\partial_{\alpha}z_1^\circ|_{C^\delta})^{1/\delta}$.
Finally, for all $\alpha\in I\setminus B_{2r}(\partial I)$, simply
$$\int_{I(\alpha)}|\partial_{\alpha}^kc(\alpha')|\dif\alpha'
\leq\|\partial_{\alpha}^kc\|_{L^1}\frac{c(\alpha)}{\eta/2},$$
because $c(\alpha)\geq\eta/2$.
\end{proof}

From now on we fix the parameters $\eta$ and $s$ satisfying $\eta(2+s^\delta)<1$ with $s<1/3$.
We remark that they are not necessarily very small (e.g.~we can take $\eta=s=\tfrac{1}{4}$ for $\delta=1$).

\section{The pseudo-interface}\label{sec:z}

We define our pseudo-interface $z$ as the solution of the integro-differential equation
\begin{equation}\label{eq:z}
\begin{split}
\partial_tz&=F(t,z^\circ,z),\\
z|_{t=0}&=z^\circ,
\end{split}
\end{equation}
given by the operator
$$F:=\psi_0E+\psi_1 E^{1)}-(t\kappa+i(tD^{(0)}\cdot\partial_{\alpha}(c\tau)+h\psi_1)\partial_{\alpha}z^{\circ}),$$
where 
$\tau$ is given in \eqref{t},
$\{\psi_0,\psi_1\}$ is the partition of the unity we fixed in \eqref{partition} and $c$ the growth-rate \eqref{c}.
The operator $E(t,z^\circ,z)$ extending $B$ outside $tc(\alpha)=0$ was already introduced
in \eqref{def:E2},
$$E:=\sum_{b=\pm}B_{b,b}.$$
Thus, it expands as
$$E(t,\alpha)
:=\frac{1}{2\pi}\sum_{b=\pm}\int_{\T}\left(\frac{1}{z_{b}(t,\alpha)-z_{b}(t,\beta)}\right)_1(\partial_{\alpha}z_{b}(t,\alpha)-\partial_{\alpha}z_{b}(t,\beta))\dif\beta.$$
The term $E^{1)}(t,z^\circ)$ is 
$$E^{1)}:=E^{(0)}+tE^{(1)},$$
where $E^{(0)}(z^\circ)$ and $E^{(1)}(z^\circ)$ are
\begin{equation}\label{E01}
\begin{split}
E^{(0)}(\alpha)
&:=\frac{1}{2\pi}\sum_{b=\pm}\int_{\T}\left(\frac{1}{z^\circ(\alpha)-z^\circ(\beta)}\right)_1(\partial_{\alpha}z^\circ(\alpha)-\partial_{\alpha}z^\circ(\beta))\dif\beta,\\
E^{(1)}(\alpha)
&:=\frac{1}{2\pi}\sum_{b=\pm}\int_{\T}\left(\frac{1}{z^\circ(\alpha)-z^\circ(\beta)}\right)_1(\partial_{\alpha}z_b^{(1)}(\alpha)-\partial_{\alpha}z_b^{(1)}(\beta))\dif\beta\\
&\,\,-\frac{1}{2\pi}\sum_{b=\pm}\int_{\T}\left(\frac{z_b^{(1)}(\alpha)-z_b^{(1)}(\beta)}{(z^\circ(\alpha)-z^\circ(\beta))^2}\right)_1(\partial_{\alpha}z^\circ(\alpha)-\partial_{\alpha}z^\circ(\beta))\dif\beta,
\end{split}
\end{equation}
with
$$z_b^{(1)}:=E^{(0)}-bc\tau^\perp.$$
The terms $\kappa(z^\circ)$ and $D^{(0)}(z^\circ)$ are
\begin{equation}\label{kappaD}
\begin{split}
\kappa(\alpha)&:=2\left(\partial_{\alpha}^2z^\circ\left(\frac{ c\tau}{(\partial_{\alpha}z^\circ)^2}\right)_1
+i\left(\frac{1}{\partial_{\alpha}z^\circ}\right)_2\partial_{\alpha}(c\tau)\right),\\
D^{(0)}(\alpha)&:=-i(c\tau+\tfrac{1}{2}).
\end{split}
\end{equation}
The term $h(t,z^\circ,z)$ is the time-dependent average defined on each connected component $(\alpha_1,\alpha_2)$ of $\{c(\alpha)>0\}$ as
\begin{equation}\label{mean}
h(t)
:=\frac{\displaystyle\int_{\alpha_1}^{\alpha_2} H\dif\alpha}
{\displaystyle\int_{\alpha_1}^{\alpha_2}\psi_1
\partial_{\alpha}z\cdot\partial_{\alpha}z^\circ\dif\alpha},
\end{equation}
where
$$H(t,\alpha):=(E-B-t\kappa)\cdot\partial_{\alpha}z^\perp
+\psi_1(E^{1)}-E)\cdot\partial_{\alpha}z^\perp
+t(D-D^{(0)}\partial_{\alpha}z\cdot\partial_{\alpha}z^\circ)\cdot\partial_{\alpha}(c\tau).$$
We notice that, although $h$ is piecewise constant, $h\psi_1$ is smooth in $\alpha$ (recall $\psi_1\lesssim c$ by Lemma~\ref {cor:monotone}).

As we will see in the next lemmas, $E^{(0)}=E|_{t=0}$, $E^{(1)}=\partial_tE|_{t=0}$, $z_b^{(1)}=\partial_tz_b|_{t=0}$ and $D^{(0)}=D|_{t=0}$. Thus, $E^{1)}$ equals to the first order expansion in time of $E$. In addition, we will see that $\kappa=\partial_t(E-B)|_{t=0}$.

In the rest of this section we will assume that there exists a solution $z$ of equation \eqref{eq:z} satisfying \eqref{hyp:z}-\eqref{equiCA} and show that this implies that $z$ satisfies \eqref{bcond:1}\eqref{cond:z} as well.

\begin{thm}\label{thm:subsolution}
Let $z^\circ\in H^{k_\circ}(\T;\R^2)$ be a closed chord-arc curve with $k_\circ\geq 6$.
Assuming that, for some $T>0$, there exists $z\in C_tH^{k_\circ-2}$ with $\partial_tz\in C_tH^{k_\circ-3}$ solving \eqref{eq:z} and satisfying
\eqref{Angle}\eqref{equiCA}, then $z$ satisfies \eqref{bcond:1}\eqref{cond:z}.
\end{thm}

The proof of this theorem relies on the forthcoming Lemmas \ref{lemma:residue}-\ref{lemma:E1}. We start by rewriting some of the terms suitably.
Let us observe that, assuming \eqref{eq:z}, the l.h.s.~of \eqref{cond:z:1} reads, for $a=\pm$, as
\begin{equation}\label{cond:z:1:1}
\partial_tz-B_a
=(E-B_a)+\psi_1 (E^{1)}-E)-(t\kappa+i(tD^{(0)}\cdot\partial_{\alpha}(c\tau)+h\psi_1)\partial_{\alpha}z^{\circ}),
\end{equation}
and for \eqref{cond:z:2} we have
\begin{align}
&(\partial_tz-B)\cdot\partial_{\alpha}z^\perp+tD\cdot\partial_{\alpha}(c\tau)\label{cond:z:2:1}\\
&=(E-B-t\kappa)\cdot\partial_{\alpha}z^\perp
+\psi_1(E^{1)}-E)\cdot\partial_{\alpha}z^\perp
+t(D-D^{(0)}\partial_{\alpha}z\cdot\partial_{\alpha}z^\circ)\cdot\partial_{\alpha}(c\tau)-h\psi_1\partial_{\alpha}z\cdot\partial_{\alpha}z^{\circ}.\nonumber
\end{align}
The core of the section is to prove that $E$ remains close to $B$ in $L^\infty$. In Lemma \ref{lemma:E-Ba} we show that $E-B_a=O(t(c+|\partial_{\alpha}c|))$. This is sufficient for \eqref{cond:z:1}.  Indeed, we show in Lemma \ref{lemma:E-B} that $E-B-t\kappa=O(t^2(c+|\partial_{\alpha}c|))$, which is sufficient for \eqref{cond:z:2}. 
In particular, this implies that $\kappa=\partial_t(E-B)|_{t=0}$ (a similar term appears in \cite{ForsterSzekelyhidi18,NoisetteSzekelyhidi20}).
Both estimates are based on a nice technical consequence of the argument principle Lemma \ref{lemma:residue} (a related argument appears in ~\cite[Lemma 4.2]{MengualSzekelyhidipp}). In Lemma \ref{lemma:D} we show that $D-D^{(0)}\partial_{\alpha}z\cdot\partial_{\alpha}z^\circ=O(t)$. The term $h$ has been introduced because
\eqref{cond:z:2} reads as
$$\int_{\alpha_1}^{\alpha}(H-h\psi_1\partial_{\alpha}z\cdot\partial_{\alpha}z^{\circ})\dif\alpha'=c(\alpha)o(t).$$
This requires to obtain a cancellation for $\alpha=\alpha_2$, which is equivalent to
 \eqref{mean}.
Finally, we show in Lemma \ref{lemma:E1} that $h$, $E-E^{1)}=O(t^2)$. All the ingredients are ready to  prove Theorem \ref{thm:subsolution}.

In the proof of Lemmas \ref{lemma:residue}-\ref{lemma:E1} all the assumptions in Theorem \ref{thm:existsz} are valid.
We start with the technical lemma.

\begin{lemma}\label{lemma:residue}
For all $k\in\N$, the function
\begin{equation}\label{C}
C_{\lambda,\mu}^k(t,\alpha):=\int_{\T}\frac{\beta^{k-1}}{(z_{\lambda}(t,\alpha)-z_{\mu}(t,\alpha-\beta))^k}\dif\beta,
\end{equation}
is uniformly bounded on $c(\alpha)>0$, $0<t\leq T$ and $\lambda,\mu\in[-1,1]$ with $\lambda\neq\mu$.
\end{lemma}
\begin{proof}
First of all, by adding and subtracting $\partial_{\alpha}z_{\mu}'/\partial_{\alpha}z_{\mu}$ (recall Sec.~\ref{sec:Intro} Notation) we split
\begin{equation}\label{residue:1}
C_{\lambda,\mu}^k(t,\alpha):=\int_{\T}\frac{\beta^{k-1}}{(z_{\lambda}-z_{\mu}')^k}\left(1-\frac{\partial_{\alpha}z_{\mu}'}{\partial_{\alpha}z_{\mu}}\right)\dif\beta
+\frac{1}{\partial_{\alpha}z_{\mu}}\int_{\T}\beta^{k-1}\frac{\partial_{\alpha}z_{\mu}'}{(z_{\lambda}-z_{\mu}')^k}\dif\beta.
\end{equation}	
The first term is controlled by $\mathcal{C}(c,z)$ and $\|\partial_{\alpha}z_{\mu}\|_{C_tC^{0,\delta}}$ (recall \eqref{equiCA}). The identity \eqref{residue:1} allows to prove the result by induction on $k$. For $k=1$, the second integral in \eqref{residue:1} is explicit (cf.~\eqref{Ind})
$$\int_{\T}\frac{\partial_{\alpha}z_{\mu}'}{z_{\lambda}-z_{\mu}'}\dif\beta
=(1+\sgn(\lambda-\mu))\pi i,$$
where we have applied the Cauchy's argument principle for $\lambda\neq\mu$ and $tc(\alpha)>0$. For $k\geq 2$, an integration by parts yields
$$
\int_{\T}\beta^{k-1}\frac{\partial_{\alpha}z_{\mu}'}{(z_{\lambda}-z_{\mu}')^k}\dif\beta
=-\frac{1}{k-1}\left(\frac{\beta}{z_{\lambda}-z_{\mu}'}\right)^{k-1}\Big|_{\beta=-\ell_\circ/2}^{\beta=+\ell_\circ/2}
+\int_{\T}\frac{\beta^{k-2}}{(z_{\lambda}-z_{\mu}')^{k-1}}\dif\beta
=S_{\lambda,\mu}^{k-1}
+C_{\lambda,\mu}^{k-1},
$$
where 
\begin{equation}\label{S}
S_{\lambda,\mu}^{j}(t,\alpha):=\frac{(-1)^{j}-1}{j}\left(\frac{\ell_\circ/2}{z_{\lambda}(t,\alpha)-z_{\mu}(t,\alpha+\ell_\circ/2)}\right)^{j}.
\end{equation}
This $S_{\lambda,\mu}^{k-1}$ is controlled by $\mathcal{C}(c,z)$ while $C_{\lambda,\mu}^{k-1}$ is bounded by induction hypothesis. Furthermore, this recursive formula for $C_{\lambda,\mu}^k$ yields ($S_{\lambda,\mu}^0=0$)
\begin{equation}\label{Residue:recurrence}
C_{\lambda,\mu}^k
=
\sum_{j=0}^{k-1}\frac{1}{(\partial_{\alpha}z_{\mu})^{k-j}}\left(\int_{\T}\beta^{j}\frac{\partial_{\alpha}z_{\mu}-\partial_{\alpha}z_{\mu}'}{(z_{\lambda}-z_{\mu}')^{j+1}}\dif\beta
+S_{\lambda,\mu}^{j}\right)
+\frac{(1+\sgn(\lambda-\mu))\pi i}{(\partial_{\alpha}z_{\mu})^{k}},
\end{equation}
which is bounded by $\mathcal{C}(c,z)$ and $\|\partial_{\alpha}z_{\mu}\|_{C_tC^{0,\delta}}$.
\end{proof}

\begin{lemma}\label{lemma:E-Ba}
For $0\leq t\leq T$ it holds that
$$E-B_{a}=O(t(c+|\partial_{\alpha}c|)).$$
\end{lemma}
\begin{proof}
Notice that $E=B=B_+=B_-$ for $tc(\alpha)=0$. Consider now $tc(\alpha)\neq 0$. 
Recall from \eqref{Bab:split} that $B_{a,b}$ is split into
$$B_{a,b}=
A_{a,b}
-i(a-b)t\partial_{\alpha}(c\tau)\frac{1}{2\pi}(C_{a,b}^1)_1,$$
where $C_{a,b}^1$ is given in \eqref{C} and 
$$
A_{a,b}:=\frac{1}{2\pi}\int_{\T}\left(\frac{1}{z_{a}-z_{b}'}\right)_1(\partial_{\alpha}z_{b}-\partial_{\alpha}z_{b}')\dif\beta.
$$
Then, for $b=-a$, we have
\begin{equation}\label{eq:E-Ba}
E-B_a
=B_{b,b}-B_{a,b}
=B_{b,b}-A_{a,b}+iat\partial_{\alpha}(c\tau) \frac{1}{\pi}(C_{a,b}^1)_1.
\end{equation}
The last term is $O(t|\partial_{\alpha}(c\tau)|)$ by Lemma \ref{lemma:residue}.
For the first term, the fundamental theorem of calculus yields
\begin{equation}\label{ftc}
\begin{split}
B_{b,b}-A_{a,b}
&=\frac{1}{2\pi}\int_{\T}
\left(\frac{1}{z_b-z_b'}-\frac{1}{z_a-z_b'}\right)_1
(\partial_{\alpha}z_{b}-\partial_{\alpha}z_{b}')\dif\beta\\
&=-\frac{a}{2\pi}tc\int_{\T}\int_{-1}^{1}\left(\frac{\tau^\perp}{(z_{\lambda}-z_{b}')^2}\right)_1(\partial_{\alpha}z_{b}-\partial_{\alpha}z_{b}')\dif\lambda\dif\beta.
\end{split}
\end{equation}
Let us check that we can apply the Fubini's theorem.
By using the regularity conditions \eqref{hyp:c}-\eqref{equiCA}, then
\eqref{ftc} can be bounded by
\begin{equation}\label{fubini}
tc\int_{0}^{\ell_\circ/2}\int_{0}^{1}\frac{\beta}{\beta^2+((1-\lambda)tc)^2}\dif\lambda\dif\beta
=tc\int_0^1\int_0^{\frac{\ell_\circ/2}{(1-\lambda)tc}}\frac{\beta}{1+\beta^2}\dif\beta\dif\lambda<\infty.
\end{equation}
Hence, the Fubini's theorem allows to interchange the order of integration of $\lambda$ and $\beta$ in \eqref{ftc}.
Then, by adding and subtracting $\beta\partial_{\alpha}^2z_b$ in \eqref{ftc}, we get
\begin{equation}\label{E-Ba}
\begin{split}
B_{b,b}-A_{a,b}
=&\,-\frac{a}{2\pi}tc\int_{-1}^{1}\int_{\T}\left(\frac{\tau^\perp}{(z_{\lambda}-z_{b}')^2}\right)_1(\partial_{\alpha}z_{b}-\partial_{\alpha}z_{b}'-\beta\partial_{\alpha}^2z_{b})\dif\beta\dif\lambda\\
&\,-\frac{a}{2\pi}tc\partial_{\alpha}^2z_{b}\left(\tau^\perp\int_{-1}^{1}C_{\lambda,b}^2\dif\lambda\right)_1,
\end{split}
\end{equation}
where, by applying the Taylor's theorem on the first term and Lemma \ref{lemma:residue} on the second one, we see that \eqref{E-Ba} is $O(tc)$ in terms of $\mathcal{C}(c,z)$ and $\|\partial_{\alpha}^2z_b\|_{C_tC^{0,\delta}}$.
\end{proof}

The next lemma shows that indeed $\kappa=\partial_t(E-B)|_{t=0}$.

\begin{lemma}\label{lemma:E-B}
For $0\leq t\leq T$ it holds that
$$E-B-t\kappa=O(t^2(c+|\partial_{\alpha}c|)).$$
\end{lemma}
\begin{proof}
Notice that
$$E-B=\tfrac{1}{2}(E-B_-)+\tfrac{1}{2}(E-B_+),$$
with $E-B_a$ given in \eqref{eq:E-Ba}.
We start with some auxiliary computations.
By combining \eqref{Residue:recurrence} and \eqref{E-Ba} we get, for $b=-a$,
\begin{align*}
B_{b,b}-A_{a,b}
=&\,-\frac{a}{2\pi}tc\int_{-1}^{1}\int_{\T}\left(\frac{\tau^\perp}{(z_{\lambda}-z_{b}')^2}\right)_1(\partial_{\alpha}z_{b}-\partial_{\alpha}z_{b}'-\beta\partial_{\alpha}^2z_{b})\dif\beta\dif\lambda\\
&\,-\frac{a}{2\pi}tc\partial_{\alpha}^2z_{b}\sum_{j=0,1}\left(\frac{\tau^\perp}{(\partial_{\alpha}z_{b})^{2-j}}\int_{-1}^{1}\left(\int_{\T}\beta^{j}\frac{\partial_{\alpha}z_{b}-\partial_{\alpha}z_{b}'}{(z_{\lambda}-z_{b}')^{j+1}}\dif\beta+S_{\lambda,b}^{j}\right)\dif\lambda\right)_1\\
&\,+\frac{a}{2}tc\partial_{\alpha}^2z_{b}\left(\frac{\tau}{(\partial_{\alpha}z_{b})^{2}}\int_{-1}^{1}(1+\sgn(\lambda-b))\dif\lambda\right)_1,
\end{align*}
where $S_{\lambda,b}^j$ is given in \eqref{S}.
Notice that
$$\int_{-1}^{1}(1+\sgn(\lambda-b))\dif\lambda=2(1+a).$$
Furthermore, it is easy to see that $S_{\lambda,b}^{j}-S_{\lambda,0}^{j}=O(t)$ by the regularity conditions \eqref{hyp:c}-\eqref{equiCA}. Then, by writing $\partial_{\alpha}z_b=\partial_{\alpha}z-itb\partial_{\alpha}(c\tau)$, it follows that
\begin{align*}
B_{b,b}-A_{a,b}
=&\,-\frac{a}{2\pi}tc\int_{-1}^{1}\int_{\T}\left(\frac{\tau^\perp}{(z_{\lambda}-z_{b}')^2}\right)_1(\partial_{\alpha}z-\partial_{\alpha}z'-\beta\partial_{\alpha}^2z)\dif\beta\dif\lambda\\
&\,-\frac{a}{2\pi}tc\partial_{\alpha}^2z\sum_{j=0,1}\left(\frac{\tau^\perp}{(\partial_{\alpha}z)^{2-j}}\int_{-1}^{1}\left(\int_{\T}\beta^{j}\frac{\partial_{\alpha}z-\partial_{\alpha}z'}{(z_{\lambda}-z_{b}')^{j+1}}\dif\beta+S_{\lambda,0}^{j}\right)\dif\lambda\right)_1\\
&\,+(1+a)tc\partial_{\alpha}^2z\left(\frac{\tau}{(\partial_{\alpha}z)^{2}}\right)_1\\
&\,+O(t^2c).
\end{align*}
Therefore, analogously to \eqref{ftc}\eqref{fubini}, the fundamental theorem of calculus jointly with the Fubini's theorem yield (recall \eqref{eq:E-Ba})
\begin{align*}
E-B-t\kappa
&=\tfrac{1}{2}(E-B_-)+\tfrac{1}{2}(E-B_+)-t\kappa\\
&=\frac{1}{\pi}t^2c\int_{-1}^{1}\int_{-1}^{1}\int_{\T}\left(\frac{\tau c'\tau'}{(z_{\lambda}-z_{\mu}')^3}\right)_1(\partial_{\alpha}z-\partial_{\alpha}z'-\beta\partial_{\alpha}^2z)\dif\beta\dif\lambda\dif\mu&&=:I_1\\
&-\frac{1}{2\pi}t^2c\partial_{\alpha}^2z\sum_{j=0,1}\left(\frac{(j+1)\tau}{(\partial_{\alpha}z)^{2-j}}\int_{-1}^{1}\int_{-1}^{1}\int_{\T}c'\tau'\beta^{j}\frac{\partial_{\alpha}z-\partial_{\alpha}z'}{(z_{\lambda}-z_{\mu}')^{j+2}}\dif\beta\dif\lambda\dif\mu\right)_1 &&=:I_2\\
&-it\partial_{\alpha}(c\tau)\frac{1}{\pi}(C_{-,+}^1-C_{+,-}^1)_1-2it\partial_{\alpha}(c\tau)\left(\frac{1}{\partial_{\alpha}z^\circ}\right)_2&&=:I_3\\
&+2tc\partial_{\alpha}^2z\left(\frac{\tau}{(\partial_{\alpha}z)^{2}}\right)_1
-2tc\partial_{\alpha}^2z^\circ\left(\frac{\tau}{(\partial_{\alpha}z^\circ)^{2}}\right)_1&&=:I_4\\
&+O(t^2c).
\end{align*}
For $I_1$, by adding and subtracting $c\tau$ and $\tfrac{1}{2}\beta^2\partial_{\alpha}^3z$, we split it into
\begin{align*}
I_1&=
-\frac{1}{2\pi}(tc)^2\partial_{\alpha}^3z\left(\tau^2\int_{-1}^{1}\int_{-1}^{1}C_{\lambda,\mu}^3\dif\lambda\dif\mu\right)_1
+\text{commutators},
\end{align*}
where
\begin{align*}
\text{commutators} &=\frac{1}{\pi}t^2c\int_{-1}^{1}\int_{-1}^{1}\int_{\T}\left(\frac{\tau ((c\tau)'-c\tau)}{(z_{\lambda}-z_{\mu}')^3}\right)_1(\partial_{\alpha}z-\partial_{\alpha}z'-\beta\partial_{\alpha}^2z)\dif\beta\dif\lambda\dif\mu\\
&+\frac{1}{\pi}(tc)^2\int_{-1}^{1}\int_{-1}^{1}\int_{\T}\left(\frac{\tau^2}{(z_{\lambda}-z_{\mu}')^3}\right)_1(\partial_{\alpha}z-\partial_{\alpha}z'-\beta\partial_{\alpha}^2z+\tfrac{1}{2}\beta^2\partial_{\alpha}^3z)\dif\beta\dif\lambda\dif\mu.
\end{align*}
The first term of $I_1$ is $O((tc)^2)$ by Lemma \ref{lemma:residue} and the commutators are $O(t^2c)$ in terms of $\|z\|_{C_tC^{3,\delta}}\lesssim\|z\|_{C_tH^4}$. Similarly, by adding and subtracting $c\tau$ and $\beta\partial_{\alpha}^2z$ for $I_2$, we gain commutators of order $O(t^2c)$, while the remaining term reads as
$$
-\frac{1}{2\pi}(tc)^2\partial_{\alpha}^2z\sum_{j=0,1}\left(\frac{(j+1)\tau^2\partial_{\alpha}^2z}{(\partial_{\alpha}z)^{2-j}}\int_{-1}^{1}\int_{-1}^{1}C_{\lambda,\mu}^{j+2}\dif\lambda\dif\mu\right)_1=O((tc)^2),
$$
where we have applied Lemma \ref{lemma:residue}.
For $I_3$, \eqref{Residue:recurrence} yields
$$C_{\lambda,\mu}^1
=
\frac{1}{\partial_{\alpha}z_{\mu}}\left(\int_{\T}\frac{\partial_{\alpha}z_{\mu}-\partial_{\alpha}z_{\mu}'}{z_{\lambda}-z_{\mu}'}\dif\beta
+(1+\sgn(\lambda-\mu))\pi i\right),$$
and thus
$$\frac{1}{\pi}(C_{-,+}^1
-C_{+,-}^1)_1
=O(t)-2\left(\frac{1}{\partial_{\alpha}z_\mu}\right)_2.$$
Finally, a direct computation shows that
$$\left(\frac{1}{\partial_{\alpha}z_\mu}\right)_2-\left(\frac{1}{\partial_{\alpha}z^\circ}\right)_2=O(t),$$
and also for $I_4$
$$\partial_{\alpha}^2z\left(\frac{ \tau}{(\partial_{\alpha}z)^2}\right)_1
-\partial_{\alpha}^2z^\circ\left(\frac{\tau}{(\partial_{\alpha}z^\circ)^2}\right)_1
=O(t),$$
in terms of $\|\partial_t z\|_{C_tC^{2,0}}\lesssim\|\partial_t z\|_{C_tH^3}$.
\end{proof}

The next lemmas deal with $D$, $h$ and $E^{1)}$.
\begin{lemma}\label{lemma:D}
For $0\leq t\leq T$ it holds that
$$D-D^{(0)}\partial_{\alpha}z\cdot\partial_{\alpha}z^\circ
=O(t).$$
\end{lemma}
\begin{proof}
Recall that
$D=-\frac{1}{2}(B_+-B_-)+D^{(0)}$.
Then, the statement follows from Lemma \ref{lemma:E-Ba} since
$$
B_+-B_-
=(B_+-E)+(E-B_-)=O(t),
$$
and using that $\partial_{\alpha}z\cdot\partial_{\alpha}z^\circ=1+O(t)$.
\end{proof}

\begin{lemma}\label{lemma:E1}
For $0\leq t\leq T$ it holds that
$$h,E^{1)}-E=O(t^2).$$
\end{lemma}
\begin{proof}
Recall the definition of $E^{(0)}$ and $E^{(1)}$ in \eqref{E01}.
First of all we observe that, in analogy with the Hilbert transform, 
it follows that $E^{(0)}\in H^{k_\circ-1}$ and also $E^{(1)}\in H^{k_\circ-2}$ in terms of  the chord-arc constant $\mathcal{C}(z^\circ)$ and $\|z^\circ\|_{H^{k_\circ}}$. Similarly, by differentiating $E$ in time
\begin{equation}\label{Et}
\begin{split}
\partial_tE
&=\frac{1}{2\pi}\sum_{b=\pm}\int_{\T}\left(\frac{1}{z_b(t,\alpha)-z_b(t,\beta)}\right)_1(\partial_{\alpha}\partial_tz_b(t,\alpha)-\partial_{\alpha}\partial_tz_b(t,\beta))\dif\beta\\
&-\frac{1}{2\pi}\sum_{b=\pm}\int_{\T}\left(\frac{\partial_tz_b(t,\alpha)-\partial_tz_b(t,\beta)}{(z_b(t,\alpha)-z_b(t,\beta))^2}\right)_1(\partial_{\alpha}z_b(t,\alpha)-\partial_{\alpha}z_b(t,\beta))\dif\beta,
\end{split}
\end{equation}
Theorem \ref{thm:subsolution} provides enough   regularity (recall $k_\circ\geq 6$) and validity of  chord-arc condition to obtain that $\partial_t E\in C_t H^{k_\circ-4}$ as long as $0\leq t\leq T$.
In particular, since $E^{(0)}=E|_{t=0}$,
the mean value theorem yields
\begin{equation}\label{E-E0}
\|E-E^{(0)}\|_{C_tH^{k_\circ-4}}=O(t).
\end{equation}
Hence, recalling the definition of
$h$ \eqref{mean} together with Lemmas \ref{lemma:E-B} and \ref{lemma:D}, it follows that
$h=O(t)$ as well. 
Then, by writing 
$$F-E^{(0)}
=\psi_0(E-E^{(0)})+t\psi_1 E^{(1)}-(t\kappa+i(tD^{(0)}\cdot\partial_{\alpha}(c\tau)+h\psi_1)\partial_{\alpha}z^{\circ}),$$
it follows from \eqref{E-E0} and the regularity of the remaining terms that
$
\|F-E^{(0)}\|_{C_tH^{k_\circ-4}}=O(t).
$
As a result from  \eqref{eq:z}, \eqref{E01} and \eqref{Et}, we have $\partial_t E-E^{(1)}=O(t)$. Notice this implies $E^{(1)}=\pa_t E|_{t=0}$. Therefore, by applying the fundamental theorem of calculus
$$E-E^{1)}=\int_0^t(\partial_t E(s)-E^{(1)})\dif s,$$
we get $E-E^{1)}= O(t^2)$. Finally,  this is enough to update the estimate for $h$ to $O(t^2)$ as well.
\end{proof}

\begin{proof}[Proof of Theorem \ref{thm:subsolution}]
\textit{Proof of \eqref{bcond:1}}. On $tc(\alpha)=0$ we directly have that 
$\partial_tz=E=B.$ \\
\textit{Proof of \eqref{cond:z}}. Consider now $tc(\alpha)>0$.
For \eqref{cond:z:1},  the expression \eqref{cond:z:1:1}  and a direct use of Lemmas \ref{lemma:E-Ba} and  \ref{lemma:E1} yield that
$$
\partial_tz-B_a=O(t).
$$
For \eqref{cond:z:2},   we use the expression \eqref{cond:z:2:1}. Then, Lemmas \ref{lemma:E-B}-\ref{lemma:E1} imply that
\begin{align*}
\left|\int_{\alpha_1}^{\alpha}(E-B-t\kappa)\cdot\partial_{\alpha}z^\perp\dif\alpha '\right|
&\lesssim
t^2\int_{\alpha_1}^{\alpha}(c+|\partial_{\alpha}c|)\dif\alpha ',\\
\left|\int_{\alpha_1}^{\alpha}\psi_1(E^{1)}-E)\cdot\partial_{\alpha}z^\perp\dif\alpha '\right|
&\lesssim t^2\int_{\alpha_1}^{\alpha}\psi_1\dif\alpha ',\\
\left|\int_{\alpha_1}^{\alpha}t(D-D^{(0)}\partial_{\alpha}z\cdot\partial_{\alpha}z^\circ)\cdot\partial_{\alpha}(c\tau)\dif\alpha '\right|
&\lesssim t^2\int_{\alpha_1}^{\alpha}(c+|\partial_{\alpha}c|)\dif\alpha',\\
\left|\int_{\alpha_1}^{\alpha}h\psi_1\partial_{\alpha}z\cdot\partial_{\alpha}z^\circ\dif\alpha '\right|
&\lesssim t^2\int_{\alpha_1}^{\alpha}\psi_1\dif\alpha ',
\end{align*}
uniformly on $c(\alpha)>0$. Firstly, recall that $\psi_1\lesssim c$ by Lemma~\ref {cor:monotone}.
Secondly,  if  $\alpha$ is closer to $\alpha_1$,
Corollary \ref{cor:monotone}  controls the integrals $\int_{\alpha_1}^\alpha c\dif\alpha'$ and $\int_{\alpha_1}^\alpha |\partial_\alpha c|\dif\alpha'$ in terms of $c(\alpha)$. Hence, the four terms above are  $O(t^2c(\alpha))$.

If $\alpha$ is closer to $\alpha_2$, Corollary \ref{cor:monotone} yields control on $\int_{\alpha}^{\alpha_2} c \dif\alpha'$ and
$\int_{\alpha}^{\alpha_2} |\partial_\alpha c| \dif\alpha'$ in terms of $c(\alpha)$. However,
 by  \eqref{mean}, it holds that $$\int_{\alpha_1}^{\alpha}(H-h\psi_1\partial_{\alpha}z\cdot\partial_{\alpha}z^\circ)\dif\alpha'=-\int_{\alpha}^{\alpha_2}(H-h\psi_1\partial_{\alpha}z\cdot\partial_{\alpha}z^\circ)\dif\alpha',$$
and thus we  can integrate on  $(\alpha,\alpha_2)$.  Therefore, for
 all $\alpha \in (\alpha_1,\alpha_2)$ the full expression is $O(t^2c(\alpha))$. Finally, by dividing by $tc(\alpha)$ we have proven that  \eqref{cond:z:2} holds. 
\end{proof}

\section{Existence of $z$}\label{sec:existence}

\begin{thm}\label{thm:existsz}
Let $z^\circ\in H^{k_\circ}(\T;\R^2)$ be a closed chord-arc curve with $k_\circ\geq 6$.
Then, there exists $z\in C_tH^{k_\circ-2}$ with $\partial_tz\in C_tH^{k_\circ-3}$ solving \eqref{eq:z} and satisfying  \eqref{hyp:z}-\eqref{equiCA} for some $0<T\ll 1$ depending on the chord-arc constant $\mathcal{C}(z^\circ)$ and the norm $\|z^\circ\|_{H^{k_\circ}}$.
\end{thm}

\begin{rem}\label{rem:regularity}
The initial regularity required $k_\circ=6$ is due to the fact that the energy estimates are easier in $H^4$, some estimates in the proof of Lemmas \ref{lemma:E-B}-\ref{lemma:E1} and that the pseudo-interface loses two derivatives on the mixing region as in \cite{ForsterSzekelyhidi18,NoisetteSzekelyhidi20}.
\end{rem}

We split the proof of this theorem into two parts. Firstly, we obtain a priori energy estimates for the equation \eqref{eq:z}. Secondly, we explain briefly how \eqref{eq:z} is regularized in order to use the a priori estimates to show the existence of the desired solution $z$.	

\subsection{A priori energy estimates}

We will take our energy as
\begin{equation}\label{def:energy}
\mathcal{E}(z):=\|z\|_{H^{k_\circ-2}}^2
+\mathcal{A}(z)^{-1}
+\mathcal{C}(z)+\mathcal{S}(z)^{-1},
\end{equation}
where $\mathcal{A}(z)$, $\mathcal{C}(z)$ are the angle and the chord-arc constants given in \eqref{A}, \eqref{CA} respectively, and
$$\mathcal{S}(z):=\inf\{\sigma(\alpha)\,:\,\alpha\in\mathrm{supp}\,\psi_0\}$$
measures the RT-stability of $z$ on $\mathrm{supp}\,\psi_0$ (recall $\sigma=(\rho_+-\rho_-)\partial_{\alpha}z_1$ with $\rho_{\pm}=\pm 1$). Notice that $\mathcal{C}(z^\circ)<\infty$ by hypothesis and that $\mathcal{A}(z^\circ)=1$, $\mathcal{S}(z^\circ)\geq 2\eta(1-2s)>0$ by construction (recall Lemma \ref{lemma:stable}). It turns out that $\frac{d}{dt}(\mathcal{A}(z)^{-1}+\mathcal{C}(z)+\mathcal{S}(z)^{-1})$ is a lower order term w.r.t.~$\mathcal{E}(z)$. Analogous terms to $\mathcal{C}$ and $\mathcal{S}$ are rigorously analyzed in \cite{CCG11,CCFG13}. The term $\mathcal{A}$ can be treated with a similar technique.

Next, we analyze the Sobolev norm of \eqref{def:energy}.
Given $0\leq k\leq k_\circ-2$, we split
\begin{align*}
\frac{1}{2}\frac{d}{dt}\|\partial_{\alpha}^kz\|_{L^2}^2
&=\int_{\T}\partial_{\alpha}^kz\cdot\partial_{\alpha}^kF\dif\alpha\\
&=\int_{\T}\partial_{\alpha}^kz\cdot\partial_{\alpha}^k(\psi_0 E)\dif\alpha &&=:I\\
&\,+\int_{\T}\partial_{\alpha}^kz\cdot\partial_{\alpha}^k(\psi_1 E^{1)} -(t\kappa +itD^{(0)}\cdot\partial_{\alpha}(c\tau)\partial_{\alpha}z^{\circ}))\dif\alpha &&=:I_{\circ}\\
&\,-\int_{\T}\partial_{\alpha}^kz\cdot\partial_{\alpha}^k(i\psi_1\partial_{\alpha}z^{\circ})h\dif\alpha &&=:I_h
\end{align*}
We claim that the terms $I$, $I_\circ$ and $I_h$ are controlled from above in terms of $\|z^\circ\|_{H^{k_\circ}}$, $\|c\|_{H^{k_\circ-1}}$, $\|\psi_j\|_{H^{k_\circ-2}}$ for $j=0,1$, $\|z\|_{H^{k_\circ-2}}$, $\mathcal{A}(z)^{-1}$, $\mathcal{C}(z)$ and $\mathcal{S}(z)^{-1}$.

The term $I_\circ$ is controlled because $\psi_1$, $E^{1)}$, $\kappa$, $D^{(0)}$, $c$ and $\tau$ only depends on $z^\circ$. Indeed, it is clear that $\psi_1$  and $c$ are smooth by definition \eqref{partition}\eqref{c}, while $\tau$ and $D^{(0)}$ lose one derivative and $\kappa$ loses two (recall \eqref{t}\eqref{kappaD}). In analogy with the Hilbert transform (recall \eqref{E01}) it follows that $E^{(0)}$ loses one derivative and similarly $E^{(1)}$ loses two, namely $$\|E^{1)}\|_{H^{k_\circ-2}}\lesssim \|z^\circ\|_{H^{k_\circ}} +\|z^\circ\|_{H^{k_\circ}}^q,$$
in terms of $\mathcal{C}(z^\circ)$ and $\|c\|_{H^{k_\circ-1}}$, for some $q\in\N$.

The term $I_h$ is controlled because $h(t)$ does not depend on $\alpha$. As we saw in Lemmas \ref{lemma:E-Ba}-\ref{lemma:E1}, it follows that $\|h\|_{L^\infty}$ is controlled in terms of $\|z\|_{H^3}$, $\|z^\circ\|_{H^4}$, $\mathcal{A}(c,z)^{-1}$ and $\mathcal{C}(c,z)$. These quantities \eqref{Angle} and \eqref{equiCA} are controlled by $\mathcal{A}(z)^{-1}$ and $\mathcal{C}(z)$ for small times (cf.~Lemma \ref{geometriclemma}).

The term $I$ is expected to be controlled because at the linear level 
$$\psi_0E\sim
-\sum_{b=\pm}\psi_0\sigma_b\Lambda z_b,$$
where $\Lambda:=(-\Delta)^{1/2}$ and $\sigma_b=(\rho_+-\rho_-)(\partial_{\alpha}z_b)_1$,
which satisfies
$\psi_0\sigma_b\geq 0$
for small times as our energy controls $\mathcal{S}(z)^{-1}$ (see the next subsection for a detailed explanation).

\subsubsection{Analysis of $I$}\label{sec:EEI}

As we mentioned in the introduction, the analysis of $I$ is classical for curves in the fully stable regime. 
In our case, all the terms are treated classically but the most singular one which needs further analysis.
Let us present here the  estimate for the main term and discuss the rest in the appendix. We will assume w.l.o.g.~that $\T=[-\pi,\pi]$ ($\ell_\circ=2\pi$).\\

The most singular term in $I$, for each $b=\pm$, is (recall Sec.~\ref{sec:Intro} Notation)
$$J:=\frac{1}{2\pi}\int_{\T}\psi_0\partial_{\alpha}^kz\cdot\int_{\T}
\left(\frac{1}{\delta_\beta z_b}\right)_1
\partial_{\alpha}^{k+1}\delta_\beta z_b\dif\beta\dif\alpha.$$
Since $z_b=z-ibtc\tau$, the term with $\partial_{\alpha}^{k+1}\delta_\beta(c\tau)$ is controlled by $\mathcal{C}(z)$ and $\|c\tau\|_{H^{k+1}}$. For $\partial_{\alpha}^{k+1}\delta_\beta z$, by adding and subtracting a suitable term, we split it into
$$J_\sigma+J_\Phi
:=-\frac{1}{2}\int_{\T}\left(\frac{\psi_0}{\partial_{\alpha}z_b}\right)_1\partial_{\alpha}^kz\cdot\Lambda(
\partial_{\alpha}^{k}z)\dif\alpha
+\int_{\T}\psi_0\partial_{\alpha}^kz\cdot\int_{\T}\Phi_b
\partial_{\alpha}^{k+1}\delta_\beta z\dif\beta\dif\alpha
,$$
where $\Lambda=(-\Delta)^{1/2}:H^1\rightarrow L^2$ is the operator
$$\Lambda f(\alpha)
:=\frac{1}{2\pi}\mathrm{pv}\!\int_{\T}\frac{\partial_{\beta}f(\beta)}{\tan(\tfrac{\alpha-\beta}{2})}\dif\beta
=\frac{1}{4\pi}\mathrm{pv}\!\int_{\T}\frac{f(\alpha)-f(\beta)}{\sin^2(\tfrac{\alpha-\beta}{2})}\dif\beta,$$
and $\Phi_b$ is the bounded kernel
\begin{equation}\label{commutator}
\Phi_b(t,\alpha,\beta) :=\frac{1}{2\pi}\left(\frac{1}{\delta_\beta z_b}-\frac{1}{\partial_{\alpha}z_b(2\tan(\beta/2))}\right)_1.
\end{equation}
For $J_\sigma$
we proceed as follows. Recall that $\partial_{\alpha}z_1^\circ>0$ uniformly on $\mathrm{supp}\,\psi_0$. 
Indeed, this is why we have opened the mixing zone slightly inside the stable regime.
Our energy \eqref{def:energy} allows to assume that $(\partial_{\alpha}z_b)_1>0$ on $\mathrm{supp}\,\psi_0$ for later times. 
Hence, using the C\'ordoba-C\'ordoba pointwise inequality $2f\cdot\Lambda f\geq\Lambda(|f|^2)$ (see \cite{CordobaCordoba03}) and the fact that $\Lambda$ is self-adjoint, we deduce that
\begin{align*}
4J_\sigma
&\leq
-\int_{\T}\left(\frac{\psi_0}{\partial_{\alpha}z_b}\right)_1\Lambda(|\partial_{\alpha}^{k}z|^2)\dif\alpha\\
&=-\int_{\T}\Lambda\left(\frac{\psi_0}{\partial_{\alpha}z_b}\right)_1|\partial_{\alpha}^{k}z|^2\dif\alpha
\leq
\left\|\Lambda\left(\frac{\psi_0}{\partial_{\alpha}z_b}\right)_1\right\|_{L^\infty}
\|\partial_{\alpha}^kz\|_{L^2}^2,
\end{align*}
with the first term controlled by $\mathcal{C}(z)$ and $\|z_b\|_{C^{2,\delta}}$. 
Notice that, since the evolution of $(\partial_{\alpha}z_b)_1$ is controlled in terms of our energy $\mathcal{E}$, the time of positiveness of $(\partial_{\alpha}z_b)_1$ on $\mathrm{supp}\,\psi_0$ depends just on the initial data $z^\circ$. 

The estimate of $J_\Phi$ is classical. We present it here for completeness.
By writing  $\partial_{\alpha}^{k+1}\delta_\beta z_b
=\partial_{\alpha}\partial_{\alpha}^{k}z_b
+\partial_{\beta}\partial_{\alpha}^kz_b'$,
we split
\begin{align*}
J_\Phi
=&\,-\frac{1}{2}\int_{\T}|\partial_{\alpha}^kz|^2\partial_{\alpha}\left(\psi_0\int_{\T}\Phi_b\dif\beta\right)\dif\alpha && =:L_1\\
&\,-\int_{\T}\psi_0\partial_{\alpha}^kz\cdot\left(\!\int_{\T}
\partial_{\alpha}^{k}z_b'\partial_{\beta}\Phi_b\dif\beta\right)\dif\alpha && =:L_2
\end{align*}
where we have integrated by parts w.r.t.~$\alpha$ and $\beta$ for $L_1$ and $L_2$ respectively.
On the one hand, $L_1$ is controlled because we can write
$$\int_{\T}\Phi_b\dif\beta
=\frac{1}{2\pi}\left(\mathrm{pv}\!\int_{\T}\frac{\dif\beta}{\delta_\beta z_b}\right)_1
=\frac{1}{2\pi}\Bigg(\frac{1}{\partial_{\alpha}z_b}\bigg(
\int_{\T}\frac{\delta_\beta\partial_{\alpha}z_b}{\delta_\beta z_b}\dif\beta
+\underbrace{\mathrm{pv}\!\int_{\T}\frac{\partial_{\alpha}z_b'}{\delta_b z_b}\dif\beta}_{=\pi i}\bigg)\Bigg)_1,$$
which is bounded in $C^1$ by $\mathcal{C}(z)$ and $\|z_b\|_{C^{2,\delta}}$.
On the other hand, $L_2$ is controlled because $\partial_{\beta}\Phi_b$
is bounded in terms of $\mathcal{C}(z)$ and $\|z_b\|_{C^3}$.\\

The analysis of the remaining terms is standard (see e.g.~\cite{CCG11}). For completeness, we have presented a compact version in Lemma \ref{lemma:lot}.

\subsection{Regularization}

In order to be able to apply the Picard's theorem we regularize the equation \eqref{eq:z} via
\begin{equation}\label{eq:z:reg}
\begin{split}
\partial_tz&=\phi_{\varepsilon}*F_\varepsilon(t,z^\circ,z),\\
z|_{t=0}&=z^\circ,
\end{split}
\end{equation}
in terms of the parameter $\varepsilon>0$, where
$$F_\varepsilon:=\psi_0E_\varepsilon+\psi_1 E^{1)}
-(t\kappa+i(tD^{(0)}\cdot\partial_{\alpha}(c\tau)+h\psi_1)\partial_{\alpha}z^{\circ}),$$
which agrees with $F$ except that $E$ has been replaced by
$$E_\varepsilon(t,\alpha)
:=\frac{1}{2\pi}\sum_{b=\pm}\int_{\T}\left(\frac{1}{\delta_\beta z_{b}}\right)_1\partial_{\alpha}\delta_\beta(\phi_{\varepsilon}*z_{b})\dif\beta.$$

Let us fix the open set where the Picard's theorem is applied. Firstly, let $O_k$ be the open subset of $H^k$ formed by chord-arc curves
$$O_k:=\{z\in H^k(\T;\R^2)\,:\,\mathcal{C}(z)<\infty\}.$$
Secondly, given $z^\circ\in O_{k_\circ}$ and $k\leq k_\circ$, we define the open neighborhood $O_k(z^\circ)$ of $z^\circ$ in $H^k$ as the set of curves $z\in O_k$ satisfying, for some fixed parameters $0<A,C,R,S<\infty$,
\begin{equation}\label{open:geometric}
\mathcal{A}(z)>A,
\quad\quad
\mathcal{C}(z)<C,
\quad\quad
\|z\|_{H^k}<R,
\end{equation}
and also
\begin{equation}\label{open:S}
\mathcal{S}(z)>S.
\end{equation}
From left to right, the conditions in \eqref{open:geometric} establish that the angle between $\partial_{\alpha}z$ and $\tau$ is uniformly non-perpendicular, which is necessary for our construction of the mixing zone (cf.~Rem.~\ref{Rem:diffeomorphism}), and that the chord-arc constant and the $H^k$-norm of $z$ are uniformly bounded respectively. Since we want $z^\circ\in O_k(z^\circ)$, necessarily $A<\mathcal{A}(z^\circ)=1$, $C>\mathcal{C}(z^\circ)$ and $R>\|z^\circ\|_{H^k}$. The condition \eqref{open:S} establishes that $z$ remains uniformly stable on $\mathrm{supp}\,\psi_0$. By Section \ref{sec:c} (cf.~Lemma \ref{lemma:stable}) we consider $S<2\eta(1-2s)$.

\begin{lemma}\label{lemma:energy}
Assume that there exists $z^\varepsilon\in C([0,T_\varepsilon];O_{k_\circ-2}(z^\circ))$ solving \eqref{eq:z:reg} for some $0<T_\varepsilon\ll 1$. Then, there exists $q\in\N$ satisfying
\begin{equation}\label{energy}
\frac{d}{dt}\mathcal{E}(z^\varepsilon)
\lesssim \mathcal{E}(z^\varepsilon)+\mathcal{E}(z^\varepsilon)^q,
\end{equation}
in terms of $A,C,R,S$ and $\|z^\circ\|_{H^{k_\circ}}$, but independently of $\varepsilon$.
\end{lemma}
\begin{proof}
This is totally analogous to the a priori energy estimates of the previous section (see \cite{CCG11}).
\end{proof}

\begin{proof}[Proof of Theorem \ref{thm:existsz}]
	\textit{Step 1.~Approximation sequence $z^\varepsilon$}. For all $\varepsilon>0$, a standard Picard iteration yields a time-dependent curve $z^\varepsilon\in C([0,T_\varepsilon];O_{k_\circ-2}(z^\circ))$ satisfying
	$$z^\varepsilon(t)=z^\circ+\int_0^t\phi_{\varepsilon}*F_\varepsilon(s,z^\circ,z^\varepsilon(s))\dif s.$$
	This $T_\varepsilon$ is taken in terms of the parameters defining $O_{k_\circ-2}(z^\circ)$ in such a way that the conditions \eqref{CA:0}\eqref{equiAS} hold. Thus, the operators $B_{a,b}$ (and so $h$) are well-defined.\\
	\indent As usual, the Gronwall's inequality applied to \eqref{energy} implies that  $\|z^\varepsilon\|_{C([0,T_\varepsilon];H^{k_\circ-2})}$ is uniformly bounded in $\varepsilon$. Furthermore, $z^\varepsilon$ satisfies \eqref{eq:z:reg} with $\|\partial_tz\|_{C([0,T_\varepsilon];H^{k_\circ-3})}$ uniformly bounded in $\varepsilon$. 
	We notice that our system is not autonomous but smooth in time.
	All these facts guarantee that the times of existence $T_\varepsilon$ do not vanish as $\varepsilon\to 0$, namely $T_\varepsilon\geq T> 0$. \\
\indent\textit{Step 2.~Convergence to $z$}.
By the Rellich-Kondrachov and the Banach-Alaoglu theorems, we may assume (taking a subsequence if necessary) that there exists $z\in C([0,T];O_{k_\circ-2}(z^\circ))$ such that $z^\varepsilon\rightarrow z$ in $C_tH^{k_\circ-3}$ and also $\partial_{\alpha}^{k_\circ-2}z^\varepsilon\rightharpoonup \partial_{\alpha}^{k_\circ-2}z$ as $\varepsilon\to 0$. Furthermore, $\partial_tz\in C([0,T];H^{k_\circ-3})$. Finally, it follow that $z$ solves \eqref{eq:z} and satisfies \eqref{hyp:z}-\eqref{equiCA}.
\end{proof}

\section{Proof of the main results and generalizations}\label{sec:generalizations}

In this section we glue the several proofs of the previous sections to gain clarity of how the Theorems \ref{thm:main:1} and \ref{thm:main:2} are proved. In addition, we recall how this construction is generalized for piecewise constant coarse-grained densities.

\subsection{Proof of Theorems \ref{thm:main:1} and \ref{thm:main:2}}

First of all we construct the growth-rate $c$ and the partition of the unity $\{\psi_0,\psi_1\}$ as in \eqref{c} and \eqref{partition} respectively, in terms of $z^\circ$ and some small parameters $\eta,s$ (e.g.~$\eta=s=\tfrac{1}{4}$, $\delta=1$). By Lemma \ref{lemma:c}, this $c$ satisfies the inequality \eqref{c:hull} and the regularity condition \eqref{hyp:c} (indeed $c\in C^\infty$).

Once these functions are fixed, the Theorem \ref{thm:existsz} implies the existence of a time-dependent pseudo-interface $z$ satisfying the equation \eqref{eq:z} and the regularity conditions \eqref{hyp:z}-\eqref{equiCA} for some $T\ll 1$. By Theorem \ref{thm:subsolution}, this $z$ satisfies the growth conditions \eqref{bcond:1}\eqref{cond:z}.

Next, we construct the mixing zone $\Mixzone$ and the non-mixing zones $\Omega_{\pm}$ by \eqref{Mixzone}\eqref{Omegat} respectively. Then, we define the triplet $(\bar{\rho},\bar{v},\bar{m})$ by \eqref{density}\eqref{velocity}\eqref{m}. Hence, by Proposition \ref{conditions:z,gamma} and Lemma \ref{conditions:z}, $(\bar{\rho},\bar{v},\bar{m})$ is a subsolution to IPM for some $0<T'\leq T$.

Finally, the h-principle in IPM (Theorem \ref{thm:hprinciple}) yields infinitely many mixing solutions to IPM starting from \eqref{rho0}\eqref{Omega0}.\\

The proof of Theorem \ref{thm:main:2} is analogous to the one of Theorem \ref{thm:main:1}. The main difference for the asymptotically flat case is that, since the domain of integration is $\R$ instead of $\T$, most of the integrals are taken with the Cauchy's principal value at infinity. 
In this case, \eqref{Ind} reads as
\begin{equation}\label{newindex}
\begin{split}
\mathrm{Ind}_{z_+(t)}(x)
&=\tfrac{1}{2}\car{\Omega_{+}(t)}(x),\\
\mathrm{Ind}_{z_-(t)}(x)
&=\tfrac{1}{2}(1-\car{\Omega_{-}(t)}(x)),
\end{split}
\end{equation}
which changes the limits \eqref{v:limits} but not the result. 
The proof of Theorem \ref{thm:main:2} in the $x_1$-periodic case is even closer to the one of Theorem \ref{thm:main:1}. In this case \eqref{newindex} holds as well, but we do not have to deal with the infinity since the domain is $\T$.

\subsection{Piecewise constant coarse-grained densities}\label{sec:Piecewise}

Following \cite{ForsterSzekelyhidi18,NoisetteSzekelyhidi20} we split the mixing zone into several levels $L=\{\lambda_j\,:\,1\leq|j|\leq N\}$ for $N\geq 1$ with
$$\lambda_j=\sgn j\tfrac{2|j|-1}{2N-1},$$
namely we consider
$$\Mixzone^j(t):=\{z_{\lambda}(t,\alpha)\,:\,c(\alpha)>0,\,\lambda\in(-\lambda_j,\lambda_j)\},$$
with $z_\lambda$ defined as in \eqref{Mixzone:z},
which satisfies $\Mixzone^1\subset\cdots\subset \Mixzone^N=:\Mixzone$. In addition we define $\Omega_{\pm}$ as in \eqref{Omegat} (or \eqref{OmegatR}).

Analogously to \cite{ForsterSzekelyhidi18,NoisetteSzekelyhidi20}, we define the piecewise constant (coarse-grained) density as ($|L|=2N$)
\begin{equation}\label{piecewise:rho}
\bar{\rho}(t,x)
:=\frac{2}{|L|}\sum_{b\in L}\mathrm{Ind}_{z_b(t)}(x)-1,
\end{equation}
for the closed case \eqref{Omega0}, while for asymptotically flat curves \eqref{Omega0R} the definition \eqref{piecewise:rho} needs to remove the last $-1$.
Observe that $\bar{\rho}=\pm 1$ on $\Omega_{\pm}$ while $\bar{\rho}$ approaches the linear profile in \cite{Szekelyhidi12,CCFpp,CFM19} inside the mixing zone.

Analogously to \eqref{velocity}, the Biot-Savart law yields
\begin{equation}\label{v:N}
\begin{split}
\bar{v}(t,x)
&=-\left(\frac{1}{\pi i|L|}\sum_{b\in L}\int\frac{(\partial_{\alpha}z_b(t,\beta))_2}{x-z_b(t,\beta)}\dif\beta\right)^*\\
&=-\frac{1}{\pi|L|}\sum_{b\in L}\int\left(\frac{1}{x-z_b(t,\beta)}\right)_1\partial_{\alpha}z_b(t,\beta)\dif\beta,
\quad\quad
x\neq z_b(t,\beta).
\end{split}
\end{equation}

Analogously to \eqref{m}, we write the relaxed momentum as
$$\bar{m}:=\bar{\rho}\bar{v}-(1-\bar{\rho}^2)(\gamma+\tfrac{1}{2}i),$$
in terms of some
$$\gamma:=\sum_{j=1}^N\nabla^\perp g_j\car{\Mixzone^j},$$
with $g_j(t,x)$ to be determined. Hence, analogously to \eqref{G}, we define $g_j$ in $(\alpha,\lambda)$-coordinates as
$$G_j(t,\alpha,\lambda)
:=\int_{\alpha_1}^{\alpha}\left(\sum_{a=\pm \lambda_j}\frac{\lambda+a}{2}(\partial_tz-B_a)\cdot\partial_{\alpha}z_a^\perp-\frac{1}{N}(\lambda_ jc\tau+\tfrac{1}{2})\cdot\partial_{\alpha}z_{\lambda}\right)\dif\alpha',$$
where
$$B_a:=\sum_{b\in L}B_{a,b}.$$
Finally, using that
$$\frac{1}{N}\sum_{j=1}^N\lambda_j=\frac{N}{2N-1},$$
the condition $|\gamma|<\tfrac{1}{2}$
yields the more general regime for $c$ given in \eqref{c:regime:optimal} as $N\to\infty$ (cf.~\eqref{|gamma|}) . The rest follows analogously to the case $N=1$ (see \cite{ForsterSzekelyhidi18,NoisetteSzekelyhidi20}).
\bigskip

\textbf{Acknowledgements}.
AC, DF and FM 
acknowledge financial support from the Spanish Ministry of Science and Innovation through the Severo Ochoa Programme for Centres of Excellence in R\&D (CEX2019-000904-S) and the ICMAT Severo Ochoa grant SEV2015-0554.
AC is partially supported by the MTM2017-89976-P and
the Europa Excelencia program ERC2018-092824. 
DF is partially supported by the Line of excellence for University Teaching Staff between CM and UAM.
DF and FM are partially supported by the ERC Advanced Grant 834728 and
by the MTM2017-85934-C3-2-P.

\appendix
\section{The pressure}

\begin{lemma}\label{lemma:p}
Let $(\rho,v)$ be a mixing solution from Theorem \ref{thm:main:1} or \ref{thm:main:2}. Then, there exists a pressure $p$ satisfying the Darcy's law
\begin{equation}\label{IPM:weak:3:p}
\int_0^t\int_{\R^2}((v+\rho i)\cdot\Phi-p\nabla\cdot\Phi)\dif x\dif s=0,
\end{equation}
for every test function $\Phi\in C_c^2(\R^3;\R^2)$. Observe that \eqref{IPM:weak:3:p} agrees with \eqref{IPM:weak:3} for $\Phi=\nabla^\perp\phi$.\\ Moreover, $v=\nabla^\perp\psi$ with $p$ and $\psi$ the continuous functions given by
\begin{align*}
(p+i\psi)(t,x)
&=\frac{1}{2\pi}\sum_{b=\pm}\int\log|x-z_b(t,\beta)|\partial_{\alpha}z_b(t,\beta)^*\dif\beta\\
&+\frac{1}{2\pi i}\int_{\Mixzone(t)}\frac{1}{x-y}(\rho-\bar{\rho})(t,y)\dif y.
\end{align*}
The first term corresponds to the macroscopic contribution of $\bar{\rho}$ (cf.~\eqref{bar{p}}). The second one is the fluctuation coming from $\rho-\bar{\rho}$ and vanishes outside $\Mixzone$ (cf.~\eqref{p-bar{p}}).\\
Furthermore, for any fixed $\mathscr{E}\in C(\R_+;\R_+)$ with $\mathscr{E}(r)>0$ for $r>0$, we can select these (infinitely many) mixing solutions satisfying
\begin{equation}\label{p-bar{p}}
|((p+i\psi)-(\bar{p}+i\bar{\psi}))(t,x)|
\leq\mathscr{E}(\mathrm{dist}((t,x),\Omega_+\cup\Omega_-)),
\end{equation}
where
\begin{equation}\label{bar{p}}
(\bar{p}+i\bar{\psi})(t,x)
:=\frac{1}{2\pi}\sum_{b=\pm}\int\log|x-z_b(t,\beta)|\partial_{\alpha}z_b(t,\beta)^*\dif\beta.
\end{equation}
\end{lemma}
\begin{proof}
	Notice that $\bar{v}=\nabla^\perp\bar{\psi}$ by \eqref{velocity}. In particular,
	$$\nabla(\bar{p}+i\bar{\psi})=-i\bar{\rho},$$
	in the sense of distributions. Following \cite{Szekelyhidi12,CFM19,Mengualpp} we consider the convex integration sequence $(\rho_k,v_k)\rightarrow(\rho,v)$ in $C_tL_{w^*}^\infty$. In fact, $\rho_k\rightarrow\rho$ in $C_tL_{\text{loc}}^q$ for all $1<q<\infty$ (see \cite[p.~12]{Mengualpp}).
	Let us split $\rho_k=\bar{\rho}+\rho_k'$, $v_k=\bar{v}+v_k'$ and $p_k:=\bar{p}+p_k'$ for some $p_k'$ to be determined.
	By construction (see \cite[Lemma 3.1]{Mengualpp}) $(\rho_k',v_k')=(\Delta\varphi_k',-\nabla^\perp\partial_1\varphi_k')$ for some real-valued function $\varphi_k'$ which is smooth and compactly supported on $\Mixzone$. Notice that $v_k'=\nabla^\perp\psi_k'$ for $\psi_k':=-\partial_1\varphi_k'$.
	Hence, $p_k$ satisfies $\nabla(p_k+i\psi_k)=-i\rho_k$ if and only if $p_k'$ satisfies
	$$\nabla(p_k'+i\psi_k')=-i\rho_k'.$$
	Therefore ($\Delta=\nabla\nabla^*$)
	$$p_k'+i\psi_k'
	=-i\nabla^*\varphi_k'
	+f_k,$$
	for some (time-dependent) entire function $f_k$. Since $\varphi_k'$ is compactly supported on $\Mixzone$, necessarily $(f_k(t,x))_2\rightarrow 0$ as $|x|\to\infty$.
	Therefore, the Liouville's theorem ($e^{if_k(t)}$ is entire and bounded) implies that $f_k$ equals to a (time-dependent) real constant. Hence, as we are choosing $p_k'$, we may assume that $f_k=0$. Finally, the Cauchy-Pompeiu's formula yields
	$$\nabla^*\varphi_k'(t,x)
	=\frac{1}{2\pi }\int_{\Mixzone(t)}\frac{1}{x-y}\rho_k'(t,y)\dif y.$$
	This concludes the proof by taking the limit $k\to\infty$. The inequality \eqref{p-bar{p}} can be guaranteed by following the proof of the quantitative h-principle in \cite{CFM19}.
\end{proof}

\section{Auxiliary lemmas}

\begin{lemma}\label{geometriclemma}
Let $z\in C([0,T];C^{1,\delta}(\T;\R^2))$.
Assume that, for some parameters $0<A,C,R,S<\infty$,
$$\mathcal{A}(z(t))>A,
\quad\quad\mathcal{C}(z(t))<C,
\quad\quad |\partial_{\alpha}z(t)|_{C^\delta}<R,
\quad\quad\mathcal{S}(z(t))>S,$$
for all $0\leq t\leq T$.
Then, there exists $0<T'(A,C,R,S,\delta,\|c\tau\|_{C^{1,\delta}})\leq T$ such that the equi-chord-arc condition holds:	
\begin{equation}\label{CA:0}
|z_{\lambda}(t,\alpha)-z_{\mu}(t,\alpha-\beta)|^2
\geq D\left(\frac{\beta^2}{(2C)^2}+((\lambda-\mu)tc(\alpha))^2\right),
\end{equation}
for all $\alpha,\beta\in\T$, $\lambda,\mu\in[-1,1]$ and $0\leq t\leq T'$, where
$D\equiv 1-\sqrt{1-(A/2)^2}$. In addition,
\begin{equation}\label{equiAS}
\sup_{\lambda\in[-1,1]}\mathcal{A}(z_\lambda(t))>A/2,
\quad\quad
\sup_{\lambda\in[-1,1]}\mathcal{S}(z_\lambda(t))>S/2.
\end{equation}
\end{lemma}
\begin{proof}
First of all notice that $A<1$, and thus $D< 1$ as well.
In particular, \eqref{CA:0} holds for $tc(\alpha)=0$. Henceforth, let $c(\alpha)>0$ and $0<t\leq T'$ for some $0<T'\leq T$ to be determined. Notice that we can take $T'$ satisfying (recall Sec.~\ref{sec:Intro} Notation)
\begin{equation}\label{CA:1}
|\delta_\beta z_\mu|
\geq |\delta_\beta z|
-t|\delta_\beta(c\tau)|\\
\geq\left(\frac{1}{C}-t|c\tau|_{C^1}\right)|\beta|
\geq\frac{|\beta|}{2C}.
\end{equation}
We split the proof into two cases, depending on the following parameter $$r\equiv\left(\frac{A}{2^3CR}\right)^{1/\delta}.$$
\indent\textit{Case $|\beta|\leq r$}. By writing,
\begin{equation}\label{CA:4}
z_{\lambda}-z_{\mu}'
=\delta_\beta z_\mu-(\lambda-\mu)tc\tau^\perp,
\end{equation}
we split the l.h.s.~of \eqref{CA:0} into
\begin{equation}\label{CA:2}
|z_{\lambda}-z_{\mu}'|^2
=|\delta_\beta z_\mu|^2+((\lambda-\mu)tc)^2
-2(\lambda-\mu)tc\delta_\beta z_\mu\cdot\tau^\perp.
\end{equation}
Let us analyze its third term. By our choice of $r$ and using $|\partial_{\alpha}z|\geq 1/C$, we can take $T'$ satisfying
$$\left|\frac{\delta_\beta z_\mu}{|\delta_\beta z_\mu|}-\frac{\partial_{\alpha}z}{|\partial_{\alpha}z|}\right|
\leq 2\frac{\left|\triangle_\beta z_\mu-\partial_{\alpha}z\right|}{|\partial_{\alpha}z|}
\leq 2C(|\partial_{\alpha}z|_{C^\delta}|\beta|^\delta+t|c\tau|_{C^1})
\leq A/2.$$
Then, by adding and subtracting $\partial_{\alpha}z/|\partial_{\alpha}z|$, we deduce that
(recall \eqref{Mixzone:z}\eqref{t})
$$\frac{|\delta_\beta z_\mu\cdot \tau|}{|\delta_\beta z_\mu|}
\geq
\frac{\partial_{\alpha}z}{|\partial_{\alpha}z|}\cdot\tau
-\left|\frac{\delta_\beta z_\mu}{|\delta_\beta z_\mu|}-\frac{\partial_{\alpha}z}{|\partial_{\alpha}z|}\right|
\geq A/2,$$
which implies that
\begin{equation}\label{CA:3}
(\delta_\beta z_\mu\cdot\tau^\perp)^2
=|\delta_\beta z_\mu|^2-(\delta_\beta z_\mu\cdot\tau)^2
\leq (1-(A/2)^2)|\delta_\beta z_\mu|^2.
\end{equation}
Finally, by applying \eqref{CA:1} and \eqref{CA:3} into \eqref{CA:2}, we deduce that
\begin{align*}
|z_{\lambda}-z_{\mu}'|^2
&\geq(1-\sqrt{1-(A/2)^2})(|\delta_\beta z_\mu|^2+((\lambda-\mu)tc)^2)\\
&\geq(1-\sqrt{1-(A/2)^2})\left(\frac{\beta^2}{(2C)^2}+((\lambda-\mu)tc)^2\right).
\end{align*}
\indent\textit{Case $|\beta|>r$}. On the one hand, by applying \eqref{CA:1}\eqref{CA:4}, the l.h.s.~of \eqref{CA:0} can be bounded from below as
$$|z_{\lambda}-z_{\mu}'|
\geq|\delta_\beta z_\mu|-2t\|c\|_{C^0}
\geq\frac{|\beta|}{2C}-2t\|c\|_{C^0}.$$
On the other hand, the r.h.s.~of \eqref{CA:0} can be bounded from above as
$$\frac{\beta^2}{(2C)^2}+((\lambda-\mu)tc)^2
\leq\frac{\beta^2}{(2C)^2}+(2t\|c\|_{C^0})^2.
$$
Thus, it is enough to guarantee that
$$\left(\frac{|\beta|}{2C}-2t\|c\|_{C^0}\right)^2
\geq D\left(\frac{\beta^2}{(2C)^2}+(2t\|c\|_{C^0})^2\right),$$
or equivalently
$$\frac{(1-\tau)^2}{1+\tau^2}\geq D
\quad\textrm{with}\quad
\tau\equiv\frac{4C\|c\|_{C^0}}{|\beta|}t.$$
Since $D\ll 1$, this holds for all $|\beta|>r$ by taking $T'$ small enough.

Finally, it is clear that \eqref{equiAS} holds for small times.
\end{proof}

\begin{lemma}\label{lemma:lot} The remaining terms of $I$ from Section \ref{sec:EEI} are lower order terms.
\end{lemma}
\begin{proof}
By combining the general Leibniz rule applied to $(\psi_0,K_b,\partial_{\alpha}\delta_\beta z_b)$ where
$$K_b(t,\alpha,\beta):=\left(\frac{1}{\delta_\beta z_b(t,\alpha)}\right)_1,$$
with the Fa\`a di Bruno's formula applied to the kernel $K_b$, we split ($j=(j_0,j_1,j_2)$)
$$I=\frac{1}{2\pi}\sum_{b=\pm}\sum_{|j|=k}
\sum_{n\in\pi_{j_1}}\binom{k}{j}(-1)^{|n|}C_nI_b(j,n),$$
where 
$$I_b(j,n)
:=\int_{\T}(\partial_{\alpha}^{j_0}\psi_0)\partial_{\alpha}^kz\cdot\int_{\T}
\left(\frac{\prod_{i=1}^{j_1}(\partial_{\alpha}^i\delta_\beta z_b)^{n_i}}{(\delta_\beta z_b)^{|n|+1}}\right)_1
(\partial_{\alpha}^{j_2+1}\delta_\beta z_b)\dif\beta\dif\alpha,$$
with $\pi_{j_1}:=\{n\in\N_0^{j_1}\,:\,n_1+2n_2+\cdots+j_1n_{j_1}=j_1\}$ and
$$C_n:=\frac{|n|!j_1!}{n_1!1!^{n_1}\cdots n_{j_1}!j_1!^{n_{j_1}}}>0.$$
The most singular term $j=(0,0,k)$ ($\Rightarrow n=0$) has been analyzed in Section \ref{sec:EEI}.\\

\textit{Second singular term}. Let us consider $j=(0,k,0)$. If $n_k=1$ then
$$I_b(j,n)=\int_{\T}\psi_0\partial_{\alpha}^kz\cdot\int_{\T}
\left(\frac{\partial_{\alpha}^k\delta_\beta z_b}{(\delta_\beta z_b)^{2}}\right)_1
\partial_{\alpha}\delta_\beta z_b\dif\beta\dif\alpha.$$
By splitting $\partial_{\alpha}\delta_\beta z_b$ into its real and imaginary part and comparing
$$\frac{(\partial_{\alpha}\delta_\beta z_b)_l}{(\delta_\beta z_b)^{2}}
\sim\frac{(\partial_{\alpha}^2z_b)_l}{(\partial_{\alpha}z_b)^2(2\tan(\beta/2))},
\quad\quad l=1,2,$$
we obtain a Hilbert transform acting on $\partial_{\alpha}^kz_b$ while the commutator is a bounded kernel as in \eqref{commutator}. If $n_k=0$, notice that for any $k\geq 3$ we have
$n_{k-1}\leq\frac{k}{k-1}<2,$
that is $n_{k-1}=0$ or $1$ (the case $k<3<k_\circ-2$ is easier). If $n_{k-1}=1$ ($\Rightarrow n_1=1$) then simply
$$|I_b(j,n)|\lesssim
\mathcal{C}(z)^{3}\|\psi_0\|_{L^\infty}\|\partial_{\alpha}^kz\|_{L^2}|\partial_{\alpha}z_b|_{C^{1}}^{2}|\partial_{\alpha}^{k-1}z_b|_{C^\delta},$$
and for $n_{k-1}=0$
\begin{equation}\label{lot:1}
|I_b(j,n)|\lesssim \mathcal{C}(z)^{|n|+1}\|\psi_0\|_{L^\infty}\|\partial_{\alpha}^kz\|_{L^2}\|z_b\|_{C^{k-2,1}}^{|n|+1}.
\end{equation}

\textit{Third singular term}. Let us consider $j=(0,1,k-1)$ ($\Rightarrow n=1$):
$$I_b(j,n)
=\int_{\T}\psi_0\partial_{\alpha}^kz\cdot\int_{\T}
\left(\frac{\partial_{\alpha}\delta_\beta z_b}{(\delta_\beta z_b)^{2}}\right)_1
(\partial_{\alpha}^{k}\delta_\beta z_b)\dif\beta\dif\alpha.$$
This is analogous to the case $(0,k,0)$. The case $j=(1,0,k-1)$ is analogous too. Let us consider now $j=(0,k-1,1)$.
If $n_{k-1}=1$ then
$$I_b(j,n)
=\int_{\T}\psi_0\partial_{\alpha}^kz\cdot\int_{\T}
\left(\frac{\partial_{\alpha}^{k-1}\delta_\beta z_b}{(\delta_\beta z_b)^{2}}\right)_1
(\partial_{\alpha}^{2}\delta_\beta z_b)\dif\beta\dif\alpha,$$
and so
$$|I_b(j,n)|\lesssim
\mathcal{C}(z)^2\|\psi_0\|_{L^\infty}
\|\partial_{\alpha}^kz\|_{L^2}
|\partial_{\alpha}^{k-1}z_b|_{C^\delta}
|\partial_{\alpha}^2z_b|_{C^{1}}.$$
If $n_{k-1}=0$ then \eqref{lot:1} holds.
The case $j=(1,k-1,0)$ is analogous.\\

\textit{Harmless terms}. For $0\leq j_1,j_2\leq k-2$, simply
$$|I_{b}(j,n)|\lesssim \mathcal{C}(z)^{|n|+1}\|\partial_{\alpha}^{j_0}\psi_0\|_{L^\infty}\|\partial_{\alpha}^kz\|_{L^2}\|z_b\|_{H^k}^{|n|+1}.$$
This concludes the proof.
\end{proof}


\bibliographystyle{abbrv}
\bibliography{partially_unstable_Bib}

\begin{flushleft}
\quad\\
\textsc{Instituto de Ciencias Matem\'aticas, CSIC-UAM-UC3M-UCM, E-28049 Madrid, Spain.}\\
\textit{E-mail address:} \href{mailto:angel_castro@icmat.es}{\nolinkurl{angel_castro@icmat.es}}
	
\quad\\
\textsc{Departamento de Matem\'aticas, Universidad Aut\'onoma de Madrid;
	Instituto de Ciencias Matem\'aticas, CSIC-UAM-UC3M-UCM, E-28049 Madrid, Spain.}\\
\textit{E-mail address:} \href{mailto:daniel.faraco@uam.es}{\nolinkurl{daniel.faraco@uam.es}}

\quad\\
\textsc{Departamento de Matem\'aticas, Universidad Aut\'onoma de Madrid; Instituto de Ciencias Matem\'aticas (CSIC-UAM-UC3M-UCM), E-28049 Madrid, Spain.}\\
\textit{E-mail address:} \href{mailto:francisco.mengual@uam.es}{\nolinkurl{francisco.mengual@uam.es}}
\end{flushleft}

\end{document}